\documentclass[12pt,twoside,titlepage]{article}
\usepackage{amsmath}
\usepackage{latexsym,amssymb}

\font\eightrm=cmr8 scaled\magstep1
\font\fiverm=cmr5 scaled\magstep1

\title{}
\author{}
\newcommand{\lastname}{}
\makeatletter

\newcommand{\rightheadline}{\hfill\sc\lastname\hfill}
\newcommand{\leftheadline}{\hfill\sc\lastname\hfill}
\newcommand{\firstheadline}{}

\def\ps@headings{\let\@mkboth\@gobbletwo
    \gdef\@oddhead{\ifnum\value{page}=1\firstheadline
                  \else\rightheadline\rm\thepage\fi}
    \gdef\@oddfoot{}
    \gdef\@evenhead{\ifnum\value{page}=1\firstheadline
                  \else\rm\thepage\leftheadline\fi}
    \gdef\@evenfoot{}
}

\def\section{\@startsection{section}{1}{\z@}{-\bigskipamount}{\medskipamount}
{\centering\bf}}

\countdef\revised=100

\newcommand{\qed}{{\unskip\nobreak\hfil\penalty50\hskip .001pt \hbox{}
          \nobreak\hfil
          \vrule height 1.2ex width 1.1ex depth -.1ex
           \parfillskip=0pt\finalhyphendemerits=0\medbreak}\rm}

\renewcommand{\lastname}{Reilly}

\def\trivlist{\vspace{-\lastskip}\medbreak
 \@trivlist \labelwidth\z@ \leftmargin\z@
\itemindent\z@   \let\@itemlabel\@empty%
\def\makelabel##1{##1}}

\def\@thmcounter#1{\noexpand\arabic{#1}}
\def\@thmcountersep{}
\def\@begintheorem#1#2{\trivlist \item[\hskip 
\labelsep{\bf #1\ #2.\quad}]\it}
\def\@opargbegintheorem#1#2#3{\it \trivlist
      \item[\hskip \labelsep{\bf #1\ #2.\quad{\rm #3}}]}

\newtheorem{theorem}{Theorem}[section]

\newtheorem{lemma}[theorem]{Lemma}
\newtheorem{corollary}[theorem]{Corollary}
\newtheorem{definition}[theorem]{Definition}

\newtheorem{example}[theorem]{Example}

\newenvironment{proof}{\begin{trivlist}\item[\hskip%
\labelsep{\bf Proof.\quad}]}%
{\hfill\qed\rm\end{trivlist}}

\def\Proof.{\rm
  \par\ifdim\lastskip<\medskipamount\vspace{-\lastskip}\fi\smallskip
           \noindent {\bf Proof.}\enspace}

\font\fiverm=cmr5 scaled\magstep1

\def\tr{{\rm tr}\,}
\def\ltr{{\rm ltr}\,}
\def\rtr{{\rm rtr}\,}

\newcommand{\arttype}{}

\def\titlepage{\thispagestyle{headings}
 \vbox{\vspace{.5truecm}}
\noindent \arttype\par
\begin{center}
{\vbox{\baselineskip=15pt\large\bf \@title\boldmath}}
  \vskip .75truecm  
{\bf \@author} 
\end{center}

\def\abstract{\baselineskip=8pt\centerline{\bf Abstract}
\vspace{.05in} \noindent\quotation}
\def\endabstract{\if@twocolumn\else\endquotation\fi}}

\begin{document}

\input pictex

\ \\[0.5cm]

\title{Self-dual Varieties and networks in the lattice of varieties of Completely Regular Semigroups}
\author{Norman R. Reilly}
\titlepage

\vskip1.1cm

{\footnotesize
\begin{abstract}
\noindent
The kernel relation $K$ on the lattice $\mathcal{L}(\mathcal{CR})$ of varieties 
of completely regular semigroups has been a central component in many 
investigations into the structure of $\mathcal{L}(\mathcal{CR})$.  However, 
apart from the $K$-class of the trivial variety, which is just the lattice of 
varieties of bands, the detailed structure of kernel classes has remained a mystery until recently.  
Kad'ourek [RK2] has shown that for two large 
classes of subvarieties of $\mathcal{CR}$ their kernel classes are singletons.
Elsewhere (see [RK1], [RK2], [RK3]) we have provided a detailed analysis of the kernel classes of 
varieties of abelian groups.  Here we study more 
general kernel classes.  We begin with a careful development of the concept of duality in 
the lattice of varieties of completely regular semigroups and then show that the kernel 
classes of many varieties, including many self-dual varieties, of completely 
regular semigroups contain 
multiple copies of the lattice of varieties of bands as sublattices.  \\
\end{abstract}
}
\vskip1cm

\noindent
{\bf Key Words and Phrases}  Semigroups, lattices, varieties, completely regular.  \\

\noindent
{\bf AMS Mathematics Subject Classification 2010}: 20M05, 20M07, 20M17\\

\section{Introduction}

Following the presentation of necessary background in Section 2, we introduce 
in Section 3 the related concepts of the mirror image of a word in the free unary semigroup, 
the dual of a semigroup or variety of completely regular semigroups and self-dual varieties.  
The concept of duality in semigroups appears implicitly in [CP], while the concept of the dual semigroup
is explicitly defined in [L].  This leads to the concepts of the mirror image or reverse of a word, dual 
identities and dual semigroup varieties.  Petrich [Pe2007], [Pe2015a] extended these concepts to 
completely regular semigroups.   Here we take a slightly different 
approach from Petrich and take as our starting point a completely unambiguous definition of 
the mirror image of a word in $Y^+$, the free semigroup on a set $X$ together 
with the symbols $($ and $)^{-1}$.  This incorporates the definitions 
introduced by Clifford and Preston, and Lallement and  
restricts to exactly what we want in the free unary semigroup $U_X$ and is mapped on 
to exactly the relation that we want in the free completely regular semigroup 
$CR_X = U_X/\zeta$ under the natural mapping $u \longrightarrow u\zeta$.  
We develop some basic results specific to completely regular semigroup 
varieties.\\

Of particular interest to us are the relations $K_{\ell} = K\cap T_{\ell}$ and its dual 
$K_r = K\cap T_r$.   These give us a finer dissection of any $K$-class into $K_{\ell}$ and 
$K_r$-classes.  As these are complete congruences, all the classes are intervals of the form 
$[\mathcal{V}_{K_{\ell}}, \mathcal{V}^{K_{\ell}}]$ and $[\mathcal{V}_{K_r}, \mathcal{V}^{K_r}]$, 
respectively, which leads us to the operators $\mathcal{V} \longrightarrow \mathcal{V}^{K_{\ell}}$ 
and $\mathcal{V}\longrightarrow \mathcal{V}^{K_r}$.  We use the repeated application of 
these operators applied to varieties that are self-dual together with their intersections 
to generate a family of subvarieties of $\mathcal{V}K$.  In contrast with 
the similar situation using the operators $\mathcal{V} \longrightarrow \mathcal{V}^{T_{\ell}}$ 
and $\mathcal{V} \longrightarrow \mathcal{V}^{T_r}$, these varieties always constitute a sublattice 
of $\mathcal{V}K$.   In certain circumstances, we are able to characterize the 
largest variety in $\mathcal{L}(\mathcal{CR})$ whose intersection with $\mathcal{V}^K$ 
is a particular element of this sublattice.   This mimics the result due to Reilly and Zhang [RZ] 
for the particular case where $\mathcal{V} = \mathcal{T}$ and $\mathcal{V}^K = \mathcal{B}$, 
the variety of bands. 

In Sections 5 and 6 we show that there exist multiple copies of the lattice $\mathcal{L}(\mathcal{B})$ 
within certain $K$-classes in $\mathcal{L}(\mathcal{CR})$.  In Section 5 we provide a canonical 
method for generating such sublattices by applying the operators  
$\mathcal{V} \longrightarrow \mathcal{V}^{K_{\ell}}$ 
and $\mathcal{V}\longrightarrow \mathcal{V}^{K_r}$ 
repeatedly to 
three subvarieties $\mathcal{V}, 
\mathcal{V}_{\ell}, \mathcal{V}_r$ related as follows: $\mathcal{V} \subset 
\mathcal{V}_{\ell} \subset \mathcal{V}^{K_{\ell}}, \mathcal{V} \subset 
\mathcal{V}_r \subset \mathcal{V}^{K_r}$.  The varieties $\mathcal{V}_{\ell}$ and 
$\mathcal{V}_r$ can be chosen freely subject to these constraints.   This abundance of  
sublattices isomorphic to $\mathcal{L}(\mathcal{B})$ in certain kernel classes is in sharp 
contrast to the situation in Kad'ourek [RK2], in which it is shown that, for two large 
classes of subvarieties of $\mathcal{CR}$, their kernel classes are singletons.  In Section 6 
we show how it is not always necessary to employ the operators 
$\mathcal{V} \longrightarrow \mathcal{V}^{K_{\ell}}$ 
and $\mathcal{V}\longrightarrow \mathcal{V}^{K_r}$
at each level.

\section{Background}

We refer the reader to the book [PR99] for a general background on
completely regular semigroups and for all undefined notation and
terminology.  For an equivalence relation $\rho$ on a nonempty set $X$ and $x\in
X,$ $x\rho$ denotes the $\rho$-class  of $x.$ In any lattice $L,$
for $a,b \in L$ such that $a\leq b$ define the \emph{interval} $[a,b]$
to be $\{c\in L \mid a\leq c\leq b\}.$

Let $S$ be a completely regular semigroup. Then $E(S)$ denotes its set of idempotents
and $\mathcal{C}(S)$ the lattice of congruences on $S.$ For $\rho
\in \mathcal{C}(S),$ the \emph{kernel\/} and the \emph{trace\/} of
$\rho$ are given by $\ker \rho = \{a\in S \mid a\;\rho\;e \;\;\mbox{for
some}\;\; e\in E(S)\},$ $\,\tr\,\rho = \rho |_{E(S)},$ and the
\emph{left\/} and the \emph{right traces\/} of $\rho $ are $\ltr \rho
= \tr (\rho \vee \mathcal{L})^0$ and $\rtr \rho = \tr (\rho \vee
\mathcal{R})^0,$ where the join is taken within the lattice of
equivalence relations on $S.$ On $\mathcal{C}(S)$ we have several
important relations defined by
$$
\begin{array}{l}
\lambda K\rho \,\Longleftrightarrow\, \ker \lambda = \ker \rho \,,
\quad \lambda T_{\ell}\rho  \,\Longleftrightarrow\, \ltr\lambda  =
\ltr \rho \,, \\[0.2cm]
\lambda T\rho \,\Longleftrightarrow\, \tr \lambda
= \tr \rho \,, \quad \lambda T_r \rho \,\Longleftrightarrow\, \rtr
\lambda = \rtr \rho \,, \\[0.2cm]
K_{\ell} = K\cap T_{\ell}\,, \quad K_r = K\cap T_r\,, \,T = T_{\ell} \cap T_r .
\end{array}
$$
\vskip0.1cm

\noindent
Two fundamental facts concerning these relations 
are as follows (where $\epsilon$ denotes the identity relation):
$$
 K \cap T = K_{\ell} \cap K_r = \epsilon.
$$
See [PR99] for details. The classes of these relations are intervals which we write as
$$
\rho P = [\rho_P, \rho^P]\,, \qquad \mbox{for} \;P\in \{K, T_{\ell}, T_r, T,
K_{\ell}, K_r\}\,.
$$
This enables us to consider the associated operators $\rho \longrightarrow \rho_P, \rho 
\longrightarrow \rho^P$.  \\

Let $X$ be a countably infinite set. If $x_1,\ldots ,x_n \in X,$ then $w= x_1
\cdots x_n$ is a \emph{word\/} over $X,$ $h(w) = x_1$
is the \emph{head\/} of $w,$ $t(w) = x_n$ is the \emph{tail\/} of
$w,$ and $c(w) = \{x_1,\ldots ,x_n\}$ is the \emph{content\/} of $w.$
We denote by $X^+$ the free semigroup on $X$ consisting of all words
over $X.$ Let $Y = X\cup \{(\;\:,\;\,)^{-1}\}$ where
``$($'' and ``$)^{-1}$'' are two distinct elements not in $X.$ By $U_X$
denote the least subset of $Y^+$ satisfying
$$
\begin{array}{r l}
C({\rm i}) & \quad X\subseteq U_X\,, \\[0.2cm]
C({\rm ii}) & \quad u,v \in U_X \,\Longrightarrow\, uv \in U_X\,,
\\[0.2cm]
C({\rm iii}) & \quad u\in U_X \,\Longrightarrow\, (u)^{-1} \in U_X\,.
\hspace*{8cm}
\end{array}
$$

\noindent
We will often omit the subscript $X$ in $U_X$ when there is no danger
of ambiguity, and sometimes write $u^{-1}$ instead of $(u)^{-1}.$
We also write $u^0 = u(u)^{-1}$ and write $|u|_Y$ for the length of
$u$ over the alphabet $Y.$

By $C({\rm ii}),$ $U_X$ is a semigroup and by $C({\rm iii}),$ the mapping
$u\longrightarrow (u)^{-1}$ $\,(u\in U_X)$ is a unary operation on
$U_X.$ An alternative description of $U_X$ due to A.H.Clifford [C] follows.

\begin{lemma}\hspace*{-0.5cm}{\rm ([PR99] Lemma I.10.1).}
\ \,The set $U_X$ consists precisely of those elements $w$ of $Y^+$
that satisfy the following conditions{\,\rm :}\\[-0.3cm]

\noindent
{\rm (i)} the number of occurrences of $\;($ in $w$ equals the
number of occurrences of $\;)^{-1},$ \\

\noindent
{\rm (ii)} in each initial segment of $w,$ the number of occurrences
of $\;($ is at least as great as the number of occurrences of
$\;)^{-1}.$ \\

\noindent
{\rm (iii)} The symbol $\:(\:$ is never followed immediately in $w$ 
by the symbol $\;)^{-1}.$
\end{lemma}
\vskip0.1cm

\begin{lemma}\hspace*{-0.5cm}{\rm ([PR99] Lemmas I.10.3, I.10.4,
I.10.5).} \\[-0.1cm]

\noindent
{\rm (i)} Let $\iota_X: \,x\longrightarrow x$ be the embedding of
$X$ in $U_X.$ Then $(U_X, \iota_X)$ is a free unary semigroup on $X.$
\\

\noindent
{\rm (ii)} A word $w\in U$ is irreducible if and only if $w\in X$
or $w=(u)^{-1}$ for some $u\in U_X.$ \\

\noindent
{\rm (iii)} Each element $w$ of $U$ can be expressed uniquely as a
product of irreducible elements.
\end{lemma}
\vskip0.1cm

Let $S$ be a completely regular semigroup. As $S$ is a union of its
(maximal) subgroups,  we have a unary operation
$a\longrightarrow a^{-1}$ on $S$ where $a^{-1}$ is the inverse of $a$ in
the maximal subgroup of $S$ containing $a.$ Hence for the purpose
of studying varieties of completely regular semigroups, they are
considered with the binary operation of multiplication and the unary
operation of inversion. We write $a^0 = aa^{-1} \;(= a^{-1}a)$ for
any element $a$ of $S.$

The class $\mathcal{CR}$ of all completely regular semigroups
constitutes a variety. It is defined, within the class of unary
semigroups by the identities
$$
a(bc) = (ab)c\,, \quad a = aa^{-1}a\,, \quad (a^{-1})^{-1} = a\,,
\quad aa^{-1} = a^{-1}a\,.
$$
Let $x\in X$ and $\zeta$ denote the least fully invariant congruence
on $U_X$ containing the pairs
$$
(x, xx^{-1}x)\,, \quad (x,(x^{-1})^{-1})\,, \quad (xx^{-1},
x^{-1}x)\,.
$$
\vskip0.1cm

\begin{lemma}  {\rm ([C] Theorem 3.1)} 
$\zeta $ is the least unary fully invariant congruence on $U_X$ such
that $CR_X = U_X/ \zeta $ is completely regular. Moreover, $CR_X$
is a free completely regular semigroup on $X.$
\end{lemma}

For any $\mathcal{V}\in \mathcal{L}(\mathcal{CR})$ we denote by $\zeta_{\mathcal{V}}$ 
the fully invariant congruence on $CR_X$ corresponding to $\mathcal{V}$.
We denote by $\mathcal{CR}$ the variety of all completely regular
semigroups and by $\mathcal{L}(\mathcal{CR})$ the lattice of its
subvarieties. Via the usual antiisomorphism of the lattice of fully
invariant congruences on a free completely regular semigroup of
infinite rank and $\mathcal{L}(\mathcal{CR}),$ the relations $K, T, 
\,T_{\ell}$ and $T_r$ defined above transfer to the lattice of fully 
invariant congruences on $\mathcal{L}
(\mathcal{CR})$ in an obvious way.   In this context, the relations 
$K, K_{\ell}, K_r, T_{\ell}$ and $T_r$ have the important property of 
being complete congruences.   
We use the same notation for these
relations on $\mathcal{L}(\mathcal{CR})$ (and their intervals) as
for the corresponding relations on semigroups. For these operators, we
write for example, $\mathcal{V}_K,$ $\mathcal{V}^{KT_r} =
(\mathcal{V}^K)^{T_r},$ and so on. For $\mathcal{V} \in \mathcal{L}
(\mathcal{CR}),$ $\mathcal{L}(\mathcal{V})$ denotes the lattice of
all subvarieties of $\mathcal{V}.$ \\

The following varieties occur frequently:  $\mathcal{T, LZ, RZ, RB, G, R\mbox{e}G, CS}$ - the varieties of 
trivial semigroups,  left zero semigroups, right zero semigroups, rectangular bands,  groups, rectangular groups and completely simple semigroups, 
respectively, all of which are subvarieties of $\mathcal{CS}$, and $\mathcal{S, LNB, RNB, LRB, RRB,} 
\mathcal{R}e\mathcal{B}, \mathcal{SG, 
LRO, RRO, B, O}$ - 
the varieties of semilattices, left normal bands, 
right normal bands, left regular bands, right regular bands, regular bands, semilattices of groups, 
left regular 
orthogroups, right regular 
orthogroups, bands and 
orthogroups, respectively, 
all of which are varieties of completely regular semigroups containing 
$\mathcal{S}$.   \\  

\begin{lemma}
$\mathcal{CR}_P = \mathcal{CR}$ $\; (P\in \{K, T, T_{\ell}, T_r, T,
K_{\ell}, K_r\}).$
\end{lemma}

\begin{proof}
See [PR88] Theorem 4.4 and [PT] Lemma 5.11.
\end{proof}

If $\mathcal{V} \in \mathcal{L}(\mathcal{CR})$ has a basis $\{u_\alpha
= v_\alpha \}_{\alpha \in A},$ we write $\mathcal{V} = [u_\alpha
=v_\alpha ]_{\alpha \in A},$ or simply $\mathcal{V} = [u=v]$ if $A$
is a singleton. We shall sometimes write an identity $u^2=u$ as $u\in
E.$

\begin{lemma}
Let $\mathcal{V} = [u_\alpha =v_\alpha ]_{\alpha \in A} \in
[\mathcal{S}, \mathcal{CR}].$ \vspace*{0.1cm}Then
$$
\begin{array}{lcl}
\mathcal{V}^K &\!\!=\!\!& \{S\in \mathcal{CR} \mid S/ \tau \in
\mathcal{V}\} = [xu_\alpha y(xv_\alpha y)^{-1} \in E]_{\alpha \in A}
\\[0.4cm]
\mathcal{V}^T &\!\!=\!\!& \mathcal{V}^{T_{\ell}} \cap
\mathcal{V}^{T_r} = \{S\in \mathcal{CR} \mid S/ \mathcal{H}^0 \in
\mathcal{V}\} \\[0.1cm]
&\!\!=\!\!& [u_\alpha^0 = v_\alpha^0, (xu_\alpha y)^0 = (xv_\alpha
y)^0]_{\alpha \in A} \\[0.4cm]
\mathcal{V}^{T_{\ell}} &\!\!=\!\!& \{S\in \mathcal{CR} \mid S/
\mathcal{L}^0 \in \mathcal{V}\} \\[0.1cm]
&\!\!=\!\!& \mathcal{LG} \circ \mathcal{V} \\[0.1cm]
&\!\!=\!\!& [xu_\alpha =xu_\alpha (xv_\alpha )^0]_{\alpha \in
A} \\[0.4cm]
\mathcal{V}^{K_{\ell}} &\!\!=\!\!&  \mathcal{V}^K \cap
\mathcal{V}^{T_{\ell}} = \{S\in \mathcal{CR} \mid S/ (\tau \cap
\mathcal{L})^0 \in \mathcal{V}\} \\[0.1cm]
&\!\!=\!\!& [xu_\alpha y(xv_\alpha y)^{-1} \in E, xu_\alpha
=xu_\alpha (xv_\alpha )^0, xv_\alpha = xv_\alpha (xu_\alpha
)^0]_{\alpha \in A}\,.
\end{array}
$$
\end{lemma}
\vskip0.5cm

\begin{proof}
For the claims concerning $\mathcal{V}^K$ see [J], for those concering $\mathcal{V}^T$ see [J] and [R], 
for those concerning $T_r$ see [Pa], [PR90] and [Po2], for those concerning $\mathcal{V}^{K_{\ell}}$ see 
[R2].
\end{proof}

Let $\Theta $ be the set of all (nonempty) words over the alphabet
$\{T_{\ell}, T_r\}$ of the form $P_1\cdots P_n,$ where $P_1 \in
\{T_{\ell}, T_r\}$ and $P_i \neq P_{i+1}$ for $i = 1,\ldots ,n-1$
with \vspace*{0.1cm}multiplication
$$
(P_1\cdots P_m) (Q_1\cdots Q_n) = \left\{ \begin{array}{ll}
P_1 \cdots P_mQ_1\cdots Q_n & \quad \mbox{if}\;\; P_m\neq Q_1\,,
\\[0.2cm]
P_1\cdots P_mQ_2\cdots Q_n & \quad \mbox{otherwise.}
\end{array} \right.
$$
\vskip0.1cm

\noindent
Clearly $\Theta $ is a semigroup. We adjoin the \emph{empty word}
$\emptyset$ to $\Theta $ to form $\Theta^1.$

We can consider $\Theta $ as being the free semigroup $\{T_{\ell},
T_r\}^+$ on the set $\{T_{\ell},T_r\}$ modulo the relations
$T_{\ell}^2 = T_{\ell},$ $T_r^2 = T_r.$ To emphasize that $\tau \in
\Theta $ is a  word in $T_{\ell}, T_r$ we might write $\tau = \tau
(T_{\ell}, T_r).$ If we replace every occurrence of $T_{\ell}$
(respectively, $T_r$) in $\tau $ by $K_{\ell}$ (respectively, $K_r$)
then we obtain a word over $\{K_{\ell}, K_r\}$ which we will denote
by $\tau (K_{\ell}, K_r).$   Note that for each postitive integer $n$ there are exactly two 
distinct words in $\Theta$ of length $n$.  They are obtained, each from 
the other, by replacing each occurrence of $T_{\ell}$ by $T_r$ and vice 
versa.

It is easily seen that the following relation $\leq$ is a partial order on $ \Theta^1$:
$$
\sigma \leq \tau \quad \mbox{if} \quad |\sigma|  > |\tau| \quad \mbox{or}
\quad \sigma = \tau \qquad\quad (\sigma,\tau \in \Theta^1).
$$





\begin{lemma}
{\rm (i)}  $(\mathcal{V}^K)_{T_{\ell}} = \mathcal{V}^K,  
(\mathcal{V}^{T_{\ell}})_{T_r} = \mathcal{V}^{T_{\ell}}$ for $\mathcal{V} \in [\mathcal{S}, 
\mathcal{CR}]$.\\
{\rm (ii)}  $(\mathcal{V}^T)_K = \mathcal{V}^T$ for 
$\mathcal{V} \in [\mathcal{R}e\mathcal{B}, \mathcal{CR}]$.\\
{\rm (iii)} $\mathcal{V}^\sigma \cap \mathcal{V}^K = \mathcal{V}^{\sigma (K_{\ell}K_r)}\,$ for  
$\mathcal{V} \in [\mathcal{S}, \mathcal{CR}], \sigma \in \Theta^1$.
\end{lemma}
\begin{proof} (i) See [Po2], Theorem 2.4(4)  and [Po2] Theorem 1.7(3).  
(ii)  See [K]  Proposition 8.2. 
(iii)  See [R2], Proposition 3.4.
\end{proof}

\begin{lemma}
{\rm(i)}  The mapping
$$
\mathcal{V} \longrightarrow \mathcal{V}_K \qquad\quad
(\mathcal{V} \in \mathcal{L}(\mathcal{CR}))
$$
is a complete $\vee$-endomorphism of $\mathcal{L}(\mathcal{CR}).$ 

{\rm(ii)}  The mapping
$$
\mathcal{V} \longrightarrow \mathcal{V}^K \qquad\quad
(\mathcal{V} \in \mathcal{L}(\mathcal{CR}))
$$
is a complete endomorphism of $\mathcal{L}(\mathcal{CR}).$

{\rm (iii)} The mapping
$$
\mathcal{V} \longrightarrow \mathcal{V}_{T_{\ell}} \qquad\quad
(\mathcal{V} \in \mathcal{L}(\mathcal{CR}))
$$
is a complete endomorphism of $\mathcal{L}(\mathcal{CR}).$ \\

{\rm (iv)} The mapping
$$
\mathcal{V} \longrightarrow \mathcal{V}^{T_{\ell}} \qquad\quad
(\mathcal{V} \in \mathcal{L}(\mathcal{CR}))
$$
is a complete $\cap$-endomorphism of $\mathcal{L}(\mathcal{CR}).$

{\rm(v)}  The mapping
$$
\mathcal{V} \longrightarrow \mathcal{V}^{K_{\ell}} \qquad\quad
(\mathcal{V} \in \mathcal{L}(\mathcal{CR}))
$$
is a complete $\cap$-endomorphism of $\mathcal{L}(\mathcal{CR}).$ 

{\rm(vi)}  The mapping 
$$
\mathcal{V} \longrightarrow \mathcal{V}_{K_{\ell}} \qquad\quad
(\mathcal{V} \in \mathcal{L}(\mathcal{CR}))
$$
is a complete $\vee$-endomorphsm of $\mathcal{L}(\mathcal{CR}).$

\end{lemma}

\begin{proof}
(i) See Petrich/Reilly [PR88].  (ii)  See Pol\'{a}k [Po].  (iii) See Pastijn [Pa].  (iv) See Petrich and 
Reilly [PR90].  (v) and (vi) These follow immediately from the fact that $K_{\ell}$ is a complete congruence on $\mathcal{L}(\mathcal{CR}).$
\end{proof}

We gather here a few useful observations concerning the operators related to 
the $K_{\ell}$ and $K_r$ relations.

\begin{lemma}
Let $\mathcal{V} \in [\mathcal{S}, \mathcal{CR}].$
\begin{enumerate}
\item[{\rm (i)}] $(\mathcal{V}^{K_{\ell}})_{T_r} = \mathcal{V}^{K_{\ell}}$ and $
(\mathcal{V}^{K_{\ell}})_{K_{r}} = \mathcal{V}^{K_{\ell}}$.
\item[{\rm (ii)}] $\mathcal{V}^K = \bigvee\limits_{\sigma \in \Theta^1}
\mathcal{V}^{\sigma (K_{\ell},K_r)}$.
\item[{\rm (iii)}] $\mathcal{V} \subset \mathcal{V}^K
\,\Longleftrightarrow\, \mbox{either}\;\; \mathcal{V}\neq
\mathcal{V}^{K_{\ell}} \;\;\mbox{or}\;\; \mathcal{V} \neq
\mathcal{V}^{K_r}$.
\end{enumerate}
\end{lemma}
\begin{proof}
(i) By Lemma 2.6(ii) with $\sigma = K_{\ell}$, we have $\mathcal{V}^{K_{\ell}} = \mathcal{V}^K 
\cap \mathcal{V}^{T_{\ell}}$.  Then \\[-0.2cm]
\begin{eqnarray*}
(\mathcal{V}^{K_{\ell}})_{T_{r}} &=& (\mathcal{V}^K\cap \mathcal{V}^{T_{\ell}})_{T_r}\\
&=& (\mathcal{V}^K)_{T_r} \cap (\mathcal{V}^{T_{\ell}})_{T_r}\;\mbox{by the dual of Lemma 2.7(iii)}\\
&=& \mathcal{V}^K \cap \mathcal{V}^{T_{\ell}}\;\;\;\;\mbox{by Lemma 2.6(i)}\\
&=& \mathcal{V}^{K_{\ell}},
\end{eqnarray*}
which establishes the first claim.  In addition, 
$$
\mathcal{V}^{K_{\ell}} = (\mathcal{V}^{K_{\ell}})_{T_r} \subseteq (\mathcal{V}^{K_{\ell}})_{K_r} 
\subseteq \mathcal{V}^{K_{\ell}}.
$$
Therefore $(\mathcal{V}^{K_{\ell}})_{K_r} = \mathcal{V}^{K_{\ell}}$ as claimed.\\

\noindent
(ii) See [R2], Corollary 3.8. \\

\noindent
(iii) If $\mathcal{V} = \mathcal{V}^{K_{\ell}} = \mathcal{V}^{K_r},$
then $\mathcal{V}^{\sigma (K_{\ell}, K_r)} = \mathcal{V}$ for all
$\sigma \in \Theta^1$ so that, by (ii), $\mathcal{V}^K = \mathcal{V}.$
Thus the direct claim holds. Since $\mathcal{V}^{K_{\ell}} \cup
\mathcal{V}^{K_r} \subseteq \mathcal{V}^K,$  the converse is obvious.
\end{proof}

In [PT], Pastijn and Trotter show that if $\mathcal{V}$ is a proper
subvariety of $\mathcal{CR},$ then so also are $\mathcal{V}^K$ and
$\mathcal{V}^T.$ An interesting source of varieties for which
$\mathcal{V} \subset \mathcal{V}^K$ is provided below.  But first we 
need a simple technical lemma.

\begin{lemma}
Let $S \in \mathcal{CR}$ be such that $S/ \mathcal{D}$ has a least
element $D.$ For $P \in \{\mathcal{H,L,R}\},$ let $\rho_P$ be defined on $S$
by
$$
a\;\rho_P\:b \;\Longleftrightarrow\; \mbox{either}\;\; a=b
\;\;\mbox{or}\;\; a,b\in D \;\;\mbox{and}\;\; a\;P\;b\,.
$$
Then $\rho_H$ {\rm (}respectively, $\rho_L$ and $\rho_R)$ is a congruence
on $S$ such that $\rho_H \subseteq \mathcal{H},$ $D/ \rho_H \in
\mathcal{RB}$ {\rm (}respectively, $\rho_L \subseteq \mathcal{L},$ $D/
\rho_L\in \mathcal{RZ}$ and $\rho_R \subseteq \mathcal{R},$ $D/ \rho_R
\in \mathcal{LZ})$.
\end{lemma}

\begin{proof}
The proof is a simple exercise.
\end{proof}

\begin{lemma}
Let $\mathcal{V} = \mathcal{V}^T \in [\mathcal{S}, \mathcal{CR}).$  Let 
$S \in \mathcal{CR}\backslash \mathcal{V}$ and $\langle S \rangle$ denote  the 
variety generated by $S$.  \\

\noindent
{\rm (i)} Either $\langle S \rangle \cap \mathcal{V}^{K_{\ell}} \not\subseteq \mathcal{V}$ or 
$\langle S \rangle \cap \mathcal{V}^{K_r} \not\subseteq \mathcal{V}$.\\ 

\noindent
{\rm (ii)} Either $\mathcal{V} \neq \mathcal{V}^{K_{\ell}}$ or $\mathcal{V} \neq \mathcal{V}^{K_r}$.\\

\noindent
{\rm (iii)} $\mathcal{V} \neq \mathcal{V}^K$.
\end{lemma}

\begin{proof}
\hspace*{-0.1cm}Since $S \in \mathcal{CR}\backslash \mathcal{V}$, 
there must exist $u = u(x_1,\ldots ,x_n), v = v(x_1,\ldots ,x_n)\break
\in U_X$ such that $\mathcal{V}$ satisfies the identity $u=v$ but $S$ does
not. Consequently there exist $a_1,\ldots ,a_n \in S$ such that $u(a_1,
\ldots ,a_n) \neq v(a_1,\ldots ,a_n).$ The completely regular subsemigroup
of $S$ generated by $a_1,\ldots ,a_n$ must also lie in $\mathcal{CR}
\backslash \mathcal{V}.$ Hence, we may assume that $S = \langle
a_1,\ldots ,a_n\rangle .$ It then follows that $S$ has only a finite
number of $\mathcal{D}$-classes. Thus there must exist a semigroup,
which we again take to be $S,$ which has the smallest possible number
of $\mathcal{D}$-classes with $S\not\in \mathcal{V}.$ Also $S$ must
have a least $\mathcal{D}$-class, $D$ say, in the semilattice
$S/ \mathcal{D}.$

Now $\mathcal{V}$ contains the variety $\mathcal{S}$ of semilattices and $\mathcal{V}$ satisfies
the identity $u=v.$ Therefore $u$ and $v$ must have the same content.
Hence $u(a_1,\ldots ,a_n)\;\,\mathcal{D}\;\,v(a_1,\ldots,a_n).$
Let $C = D_{u(a_1,\ldots ,a_n)}.$ If $C\neq D,$ then $T =
\bigcup \{D_a \mid C\leq D_a\}$ will be a subsemigroup of $S$ with
fewer $\mathcal{D}$-classes and $T \not\in \mathcal{V}.$ This
contradicts the choice of $S.$ Hence $C=D.$ \\[-0.15cm]

\noindent
\emph{Case\/}: \,$u(a_1,\ldots ,a_n)\;\,\mathcal{H}\;\,v(a_1,\ldots
,a_n).$ Let $\rho_H$ be defined as in Lemma 2.9.  \,If $S/ \rho_H \in \mathcal{V}$ 
then since $\rho_H \subseteq
\mathcal{H},$ we have $S\in \mathcal{V}^T = \mathcal{V},$ a
contradiction. Hence $S/ \rho_H \not\in \mathcal{V}$ and there must
be another identity $p=q,$ where $p=p(x_1,\ldots ,x_m), q=
q(x_1,\ldots ,x_m) \in U_X,$ that is satisfied by $\mathcal{V}$ but
not by $S/ \rho_H.$ Let $b_1,\ldots ,b_m \in S/ \rho_H$ be such that
$p(b_1,\ldots ,b_m) \neq q(b_1,\ldots ,b_m).$ As for $u(a_1,\ldots
,a_n),$ $v(a_1,\ldots ,a_n),$ we must have $p(b_1,\ldots ,b_m),
q(b_1,\ldots ,b_m) \in D/ \rho_H.$ Since the restriction of $\mathcal{H}$
to $D/ \rho_H$ is the identity relation, we must have that either
$p(b_1,\ldots ,b_m),$ $q(b_1,\ldots ,b_m)$ are not $\mathcal{L}$-related
or they are not $\mathcal{R}$-related. This situation is
covered by the next case.\\[-0.1cm]

\noindent
\emph{Case\/}: \,Either $(u(a_1,\ldots ,a_n), v(a_1,\ldots
,a_n)) \not\in \mathcal{L}$ or $(u(a_1,\ldots,a_n),
v(a_1,\ldots ,a_n))\break
\not\in \mathcal{R}.$ It suffices to consider
the case $(u(a_1,\ldots ,a_n), v(a_1,\ldots ,a_n)) \not\in
\mathcal{L}.$ Then\break
$(u(a_1,\ldots ,a_n), v(a_1,\ldots ,a_n))
\not\in \rho_L.$ Hence $S/ \rho _L \not\in \mathcal{V}.$ However,
$D/ \rho_L \in \mathcal{RZ}.$ If the Rees quotient $S/D \in \mathcal{V}$
then we have $S/ \rho_L \in (\mathcal{RZ} \circ \mathcal{V}) \backslash
\mathcal{V},$ that is, $\mathcal{V}^{K_r} \neq \mathcal{V}.$ It
remains to show that $S/D \in \mathcal{V}.$ By way of
contradiction, suppose that $S/D \not\in \mathcal{V}.$ Then there
exist $p = p(x_1,\ldots ,x_m), q= q(x_1,\ldots ,x_m) \in U_X$
such that $\mathcal{V}$ satisfies the identity $p=q,$ but $S/D$ does 
not. Let $c_1,\ldots c_m \in S/D$ be
such that $p(c_1,\ldots ,c_m) \neq q(c_1,\ldots ,c_m).$ As previously,
$p(c_1,\ldots, c_m)$ and $q(c_1,\ldots,c_m)$ must lie in the same
$\mathcal{D}$-class, $C$ say. Moreover $C\neq D$ since $D$ is the
zero element in $S/ \mathcal{D}.$ But then $T = \bigcup \{D_a \mid
C\leq D\}$ has fewer $\mathcal{D}$-classes than $S,$ contradicting
the choice of $S.$ Hence $S/D \in \mathcal{V}$ and $\langle S \rangle \cap 
\mathcal{V}^{K_r} \not\subseteq \mathcal{V}$.  In particular
$\mathcal{V}^{K_r} \neq \mathcal{V}.$\\

By duality, this proves parts (i) and (ii).  Part (iii) then follows trivially.
\end{proof}
\vskip1cm

\section{Dual varieties}
The concept of duality in semigroups is introduced in [CP] where the
dual semigroup is implicit, while the concept of the dual semigroup
is explicitly defined in [L].  This leads to the concepts of the mirror image 
of a word, dual identities and dual varieties.  It has played an important role 
in some recent papers by Petrich [Pe2007], [Pe2015a] on varieties of 
completely regular semigroups.  In these papers, Petrich 
defines the mirror image of a ``completely regular" word directly in 
the free completely regular 
semigroup thereby raising the awkward question as to whether the 
concept is well-defined.  Here we take a slightly different approach and 
take as our starting point a completely unambiguous definition of the
mirror image of a word in $Y^+$.  This incorporates the definitions 
introduced by Clifford and Preston [CP] and Lallement [L] and  
restricts to exactly what we want in $U_X$ and is mapped onto 
exactly the relation that we want in $CR_X = U_X/\zeta$ under 
the natural mapping.

\begin{definition}
Let $w = x_1x_2\cdots x_n \in Y^+$ where $x_1,x_2,\ldots ,x_n \in
Y.$ Let $\varphi :\, Y\longrightarrow Y$ be defined by
$$
x\varphi = \left\{ \begin{array}{lcl}
x & \;\; \mbox{if} \;\; & x\in X \\[0.2cm]
)^{-1} & \;\; \mbox{if} \;\; & x = ( \\[0.2cm]
( & \;\; \mbox{if} \;\; & x = \;\,)^{-1}\,.
\end{array} \right.
$$
\vskip0.2cm

\noindent
We define the \emph{\bf mirror image\/} of $w$ to be $\overline{w}
= \varphi(x_n) \varphi (x_{n-1}) \cdots \varphi (x_1).$
\end{definition}

For example, with $w = p(q(rs)^{-1} t)^{-1}u$ where $p,q,r,s,t,u
\in X$ we have $\overline{w} = u(t(sr)^{-1}q)^{-1}p.$ 

\begin{lemma}
Let $u,v,w \in Y^+.$
\begin{enumerate}
\item[{\rm (i)}] $\overline{uv} = \overline{v}\,\overline{u}$.
\item[{\rm (ii)}] $\overline{\overline{w}} = w,$ $\overline{(w)^{-1}} = (\overline{w})^{-1}$, 
$\;\;\overline{w^0}\;\,\zeta\,\; \overline{w}^0$.
\item[{\rm (iii)}] $w\in U_X \,\Longrightarrow\, \overline{w} \in
U_x$.
\item[{\rm (iv)}] $u,v \in U_X,$ $\;\; u\;\zeta \;v
\,\Longrightarrow\, \overline{u}\;\zeta \;\overline{v}$\vspace*{0.1cm}.
\end{enumerate}
\end{lemma}

\begin{proof}
(i),(ii) \,These are self-evident. \\[-0.1cm]

\noindent
(iii) \,It is straightforward to see that since $w$ satisfies the
conditions (i), (ii), (iii) of Lemma 2.1, so also does
$\overline{w}.$ The only point worth noting is that condition (ii)
of Lemma 2.1 for $w$ is equivalent to the number of occurrences of
$\;)^{-1}$ in any final segment of $w$ being at least as great as
the number of occurrences of $\,(.$ \\

\noindent
(iv) Now $\zeta$ is generated as a semigroup congruence from the identities  
$x=x(x)^{-1}x$, $\,x=((x)^{-1})^{-1}, \,x(x)^{-1}=(x)^{-1}x.$
Since $u\:\zeta \:v$, the identity $u=v$ can be formally derived
from the above listed identities.  
Thus there exists a sequence $u = w_0,w_1,\ldots ,
w_{n+1} = v\in U_X$ such that, for each $i=0,1,\ldots ,n$ there exists an
endomorphism $\varphi_i$ of $U_X,$ an identity $p = q$ (or $q = p$) in
$I$ and $a_i, b_i \in P^+$ such that
$$
w_i = a_i \varphi_i (p) b_i\,, \quad w_{i+1} =
a_i \varphi_i(q) b_i\,.
$$

\noindent
Now consider the case where $p=x,$ $q=x(x)^{-1}x.$ Then
$$
w_i = a_i \varphi_i(x) b_i\,, \quad w_{i+1} = a_i \varphi_i(x)
(\varphi_i(x))^{-1}\varphi_i(x) b_i
$$
and we have
$$
\overline{w}_i = \overline{b}_i \overline{\varphi_i(x)}
\overline{a}_i\,, \quad \overline{w}_{i+1} = \overline{b}_i
\overline{\varphi_i(x)}\; (\overline{\varphi_i(x)})^{-1}\;
\overline{\varphi_i(x)} \overline{a_i}\,.
$$

The other possibilities for the identity $p=q$ can be dealt with
similarly. Thus $\overline{w}_i\;\, \zeta \;\, \overline{w}_{i+1}$ for
$i = 0,1,\ldots ,n-1$ and therefore $\overline{u}=\overline{w}_0\;\,\zeta
\;\, \overline{w}_{n+1} = \overline{v}.$
\end{proof}

\begin{definition}
For any semigroup $(S, \;\cdot \;)$ we define its \emph{\bf dual} $\overline{S}$ to be
$(S, *)$ where $a*b = ba$ for all $a,b\in S.$ For any variety $\mathcal{V} \in
\mathcal{L}(\mathcal{CR})$ we define the \emph{\bf dual} of $\mathcal{V}$
to be $\overline{\mathcal{V}} = \{\overline{S} \mid S\in \mathcal{V}\}$ and
we say that $\mathcal{V}$ is self-dual if $\overline{\mathcal{V}} =
\mathcal{V}.$ We denote the class of all self-dual subvarieties of
$\mathcal{CR}$ by $\mathcal{SD}.$  For any equivalence relation $\rho$ on 
$U_X$, we define the dual $\overline{\rho}$ of $\rho$ by
$u \:\overline{\rho} \:v \Longleftrightarrow \overline{u}\: \rho \:\overline{v}$.
\end{definition}

When we wish to emphasize  that an expression or formula
is to be calculated in $\overline{S}$ as opposed to $S,$ we will write
the elements as $\overline{a}, \overline{b}, \overline{c}, \ldots$.
For instance, if $u(x,y) = xy \in U_X$ and $a,b\in S,$ then $u(a,b)
= ab$ while $u(\overline{a}, \overline{b}) = a*b = ba.$ The dual
operator $S \longrightarrow \overline{S}$ reflects Green's relations
in the sense that, for $a,b \in S,$ $a\;\mathcal{L}_S\;b
\,\Longleftrightarrow\, a\;\mathcal{R}_{\overline{S}}\;b$;
$a\;\mathcal{L}_S^0\;b \,\Longleftrightarrow\,
a\;\mathcal{R}_{\overline{S}}^0\;b$.

Some elementary examples of $\overline{\mathcal{V}}$ are $\mathcal{RZ}
= \overline{\mathcal{LZ}},$  $\mathcal{RNB} = \overline{\mathcal{LNB}},$
$\mathcal{RG} = \overline{\mathcal{LG}}$ and $\mathcal{RRO} =
\overline{\mathcal{LRO}}.$ The class of self-dual varieties clearly
includes $\mathcal{T}, \mathcal{RB}, \mathcal{G}, \mathcal{CS},
\mathcal{S, NB}, \mathcal{B}, \mathcal{CR}.$

\begin{lemma}
Let $S\in \mathcal{CR},
\mathcal{U, V} \in \mathcal{L}(\mathcal{CR}), 
u, v, u_{\alpha}, v_{\alpha} (\alpha \in A) \in U_X$ 
 and 
$\mathcal{V} = [u_\alpha =v_\alpha ]_{\alpha \in A}.$
\begin{enumerate}
\item[{\rm (i)}] $\overline{S} \in \mathcal{CR}$.
\item[{\rm (ii)}] $\overline{\mathcal{V}} \in \mathcal{L}(\mathcal{CR})$.
\item[{\rm (iii)}] $\overline{\overline{\mathcal{V}}} = \mathcal{V}$.
\item[{\rm (iv)}] If \,$\mathcal{V} \subseteq \overline{\mathcal{V}}$
then\, $\mathcal{V} = \overline{\mathcal{V}}$.  
\item[{\rm (v)}]  The varieties 
$\mathcal{V}\cap \overline{\mathcal{V}}$ and\, $\mathcal{V}\vee
\overline{\mathcal{V}}$ are self-dual.
\item[{\rm (vi)}] $\overline{\mathcal{V}} = [\overline{u}_\alpha =
\overline{v}_\alpha ]_{\alpha \in A}$.
\item[{\rm(vii)}] $u \;\zeta_{\mathcal{V}}\; v \Longrightarrow \overline{u}\; \zeta_{\overline{\mathcal{V}}} 
\;\overline{v}$
\item[{\rm(viii)}] $\overline{\mathcal{U}\vee \mathcal{V}} = \overline{\mathcal{U}} \vee 
\overline{\mathcal{V}}$ and $\overline{\mathcal{U}\cap \mathcal{V}} = 
\overline{\mathcal{U}} \cap \overline{\mathcal{V}}$.
\end{enumerate}
\end{lemma}

\begin{proof}
(i)  It is clear that the dual of a group is also a group and therefore that the dual of a semigroup 
that is a union of groups is again a semigroup that is a union of groups.  Hence $\overline{S} \in 
\mathcal{CR}$.\\

\noindent
(ii) Clearly $\overline{\mathcal{V}}$ is closed under the formation
of homomorphisms, subsemigroups and direct products.  Therefore $\overline{\mathcal{V}}$ is a 
variety and, by part (i), $\overline{\mathcal{V}} \in \mathcal{L}(\mathcal{CR})$.\\[0.2cm]

\noindent
(iii) That $\overline{\overline{\mathcal{V}}} = \mathcal{V}$ follows from
the evident fact that the dual of the dual of a semigroup $(S, \;\cdot\;)$
is just $(S, \;\cdot\;)$ itself. \\[0.2cm]

\noindent
(iv) If $\mathcal{V} \subseteq \overline{\mathcal{V}},$ then we must
have $\overline{\mathcal{V}} \subseteq \overline{\overline{\mathcal{V}}}
= \mathcal{V}$ and therefore $\mathcal{V} = 
\overline{\mathcal{V}}$ as claimed. \\

\noindent
(v)  It is trivial that $\mathcal{V} \cap
\overline{\mathcal{V}}$ is self-dual. Let $S\in \mathcal{V} \vee
\overline{\mathcal{V}}.$ Then there exist $\mathcal{V}_\alpha \in
\mathcal{V} \cup \overline{\mathcal{V}},$ $\alpha \in A,$ a
subsemigroup $R$ of $\prod\limits_{\alpha \in A} S_\alpha $ and an
epimorphism $\varphi: \,R\longrightarrow S.$ It follows that
$\overline{S}$ is a homomorphic image of $\overline{R}$ which, in
turn, is a subsemigroup of $\prod\limits_{\alpha \in A}
\overline{S}_\alpha$ where $\overline{S}_\alpha \in \mathcal{V} \cup
\overline{\mathcal{V}}.$ Hence $\overline{S} \in \mathcal{V} \vee 
\overline{\mathcal{V}}.$ That \vspace*{-0.1cm}implies that
$\overline{\mathcal{V} \cup \overline{\mathcal{V}}} \subseteq
\mathcal{V} \vee \overline{\mathcal{V}}$ and therefore that
$\mathcal{V} \vee \overline{\mathcal{V}}$ is self-dual.\\[0.2cm]

\noindent
(vi) Let $a_1, a_2,\ldots ,a_n \in S\in \mathcal{V}$ and $u =
u(x_1,\ldots ,x_n) \in U_X.$ \\[-0.2cm]

\noindent
\emph{Claim\/}: \,$\overline{u} (\overline{a}_1,\ldots ,\overline{a}_n)
= u(a_1,\ldots ,a_n).$\\[-0.2cm]

We argue by induction on $|u|_Y.$ \\[-0.2cm]

\noindent
\emph{Case\/}: \,$|u|_Y = 1.$ \,Then $u=x_1 \in X.$ Hence
$\overline{u} (\overline{a}_1) = \overline{a_1} = a_1 = u(a_1).$\\

\noindent
Now assume that $|u|_Y > 1$ and that the claim is true for all $v\in
U_X$ with $|v|_Y < |u|_Y.$ Let $u = u_1 \cdots u_m$ as a product of
irreducibles (see [PR99], Lemma I.10.4).\\[-0.15cm]

\noindent
\emph{Case\/}: \,$m=1.$ \,Then $u=u_1$ is irreducible and $|u|_Y >
1.$ The only way \vspace*{0.05cm}that that can occur is if $u = (w)^{-1}$
for some $w\in U_X.$ By Lemma 3.2(ii), $\overline{(w)^{-1}} = (\overline{w})^{-1}$
so that $\overline{u} = \overline{(w)^{-1}} = (\overline{w})^{-1}$ and 
\begin{align*}
\overline{u} (\overline{a}_1,\ldots ,\overline{a}_n) =\;&
\overline{(w(\overline{a}_1,\ldots ,\overline{a}_n))^{-1}}
\\[0.25cm]
=\;& (\overline{w}(\overline{a}_1,\ldots ,\overline{a}_n))^{-1}
\\[0.25cm]
=\;& (w(a_1,\ldots ,a_n))^{-1} \qquad \mbox{by induction
hypothesis} \\[0.25cm]
=\;& u(a_1,\ldots ,a_n)\,.
\end{align*}
\vskip0.3cm

\noindent
\emph{Case\/}: \,$m > 1.$ \,Then $\overline{u} = \overline{u_1 u_2
\cdots u_m} = \overline{u}_m \cdots \overline{u}_2 \overline{u}_1$
so that from the substitution $x_i \rightarrow \overline{a}_i$
$\,(1\leq i\leq n)$ into $\overline{S}$ we \vspace*{0.2cm}obtain 
(using $\ast$ to denote the multiplication in $\overline{S}$) 
\begin{align*}
\overline{u}(\overline{a}_1, \ldots ,\overline{a}_n) =\;&
\overline{u}_m (\overline{a}_1,\ldots , \overline{a}_n) \ast \ast \ast 
\overline{u}_1 (\overline{a}_1,\ldots ,\overline{a}_n) \\[0.25cm]
=\;& \overline{u}_1 (\overline{a}_1,\ldots ,\overline{a}_n) \cdots 
\overline{u}_m (\overline{a}_1,\ldots ,\overline{a}_n) \\[0.25cm]
=\;& u_1( a_1,\ldots ,a_n) \cdots u_m (a_1,\ldots ,a_n) \qquad \mbox{by
the induction hypothesis} \\[0.25cm]
=\;& u(a_1,\ldots ,a_n)\,.
\end{align*}
\vskip0.2cm

\noindent
Thus the claim holds.

Consequently, for all $\alpha \in A,$
$$
\overline{u}_\alpha (\overline{a}_1,\ldots ,\overline{a}_n) = u_\alpha
(a_1,\ldots ,a_n) = v_\alpha (a_1,\ldots ,a_n) = \overline{v}_\alpha
(\overline{a}_1,\ldots ,\overline{a}_n)\,.
$$
Thus $\overline{S} \in [\overline{u}_\alpha = \overline{v}_\alpha
]_{\alpha \in A}$ so that $\overline{\mathcal{V}} \subseteq
{[\overline{u}_\alpha = \overline{v}_\alpha ]_{\alpha \in A}.}$
Similarly, $S\in [\overline{u}_\alpha = \overline{v}_\alpha ]_{\alpha
\in A}$ implies that
$\overline{S} \in \big[\,\overline{\overline{u}}_\alpha  =
\overline{\overline{v}}_\alpha \big]_{\alpha \in A} = [u_\alpha = v_\alpha
]_{\alpha \in A} = \mathcal{V}.$
Thus $S = \overline{\overline{S}} \in \overline{\mathcal{V}}$ whence
$[\overline{u}_\alpha = \overline{v}_\alpha ]_{\alpha \in A} \subseteq
\overline{\mathcal{V}}$ and equality prevails.\\

\noindent
(vii)  This follows immediately from (vi).\\

\noindent
(viii)  Let $u, v \in U_X$.   It follows from (vii) that
\begin{eqnarray*}
u\; \zeta_{\overline{\mathcal{U}\vee \mathcal{V}}}\; v 
&\Longleftrightarrow& \overline{u}\: 
\zeta_{\mathcal{U}\;\vee\; \mathcal{V}}\: \overline{v} \Longleftrightarrow \overline{u} \:
(\zeta_{\mathcal{U}}\;\cap\; \zeta_{\mathcal{V}})\: \overline{v}\\
&\Longleftrightarrow& \overline{u}\: \zeta_{\mathcal{U}}\: \overline{v}\;\;\mbox{and}\;\; 
\overline{u}\: \zeta_{\mathcal{V}}\: \overline{v}\: \Longleftrightarrow 
u\: \zeta_{\overline{\mathcal{U}}} \: v \:\mbox{and}\: u\: \zeta_{\overline{\mathcal{V}}} \: v \\
&\Longleftrightarrow& u\: (\zeta_{\overline{\mathcal{U}}}\: \cap\: \zeta_{\overline{\mathcal{V}}})\: 
v \: \Longleftrightarrow u\: \zeta_{\overline{\mathcal{U}}\:\vee\:\overline{\mathcal{V}}} \:v.
\end{eqnarray*}
Therefore, $\overline{\mathcal{U}\:\vee\:\mathcal{V}} = \overline{\mathcal{U}}\:\vee\: 
\overline{\mathcal{V}}$, establishing the first claim.  \\

Similarly,
\begin{eqnarray*}
u\; \zeta_{\overline{\mathcal{U}\cap \mathcal{V}}}\: v 
&\Longleftrightarrow& \overline{u} \:\zeta_{\mathcal{U}\cap \mathcal{V}} \:\overline{v}\\
&\Longleftrightarrow& \overline{u}\: (\zeta_{\mathcal{U}} \vee \zeta_{\mathcal{V}}) \:\overline{v}\\
&\Longleftrightarrow& \mbox{there exist}\: 
u_i \in U_X, \rho_i \in \{\zeta_{\mathcal{U}}, 
\zeta_{\mathcal{V}}\}\: \mbox{with}\:\\
&&  
\overline{u} = u_0 \: \rho_1 \: u_1\: \rho_2\: u_2 \ldots 
\rho_n\: u_n = \overline{v}\\
&\Longleftrightarrow& u\: \overline{\rho_1}\: \overline{u_1}\: \overline{\rho_2} \ldots 
\overline{\rho_n}\: \overline{u_n} = v\\
&\Longleftrightarrow& u\: (\zeta_{\overline{\mathcal{U}}} \vee \zeta_{\overline{\mathcal{V}}})\: v\\
&\Longleftrightarrow& u\: \zeta_{\overline{\mathcal{U}} \cap \overline{\mathcal{V}}} \: v\:. 
\end{eqnarray*}
Therefore $\overline{\mathcal{U}\cap\mathcal{V}} = \overline{\mathcal{U}} \cap \overline{\mathcal{V}}$.
\end{proof}

\begin{theorem}
{\rm (i)} $\mathcal{SD}$ is a complete sublattice of
$\mathcal{L}(\mathcal{CR}).$ \\[-0.3cm]\\
\noindent
For the remaining parts, let $\mathcal{V}
= [u_\alpha =v_\alpha ]_{\alpha \in A} \in [\mathcal{S}, \mathcal{CR}]$
be self-dual\vspace*{0.1cm}.
\begin{enumerate}
\item[{\rm (ii)}] $\overline{\mathcal{V}^{T_{\ell}}} =
\mathcal{V}^{T_r}, \quad \overline{\mathcal{V}^{K_{\ell}}} =
\mathcal{V}^{K_r}\vspace*{0.1cm}.$
\item[{\rm (iii)}] $\mathcal{V}^K$ and\, $\mathcal{V}^T$ are
self-dual\vspace*{0.1cm}.
\end{enumerate}
\end{theorem}
\begin{proof}
(i) Clearly $\mathcal{SD}$ is closed under arbitrary intersections.
Now let $S \in \mathcal{V} = \bigvee\limits_{\alpha \in A}
\mathcal{V}_\alpha $ where $\mathcal{V}_\alpha \in \mathcal{SD}$ for
all $\alpha \in A.$ Then there exist $B \subseteq A,$ $S_\alpha \in
\mathcal{V}_\alpha $ for $\alpha \in B,$ a subsemigroup $R$ of
$\prod\limits_{\alpha \in A} S_\alpha $ and an epimorphism $\varphi
: \,R\longrightarrow S.$ Since each $\mathcal{V}_\alpha $ is
self-dual, we have $\overline{S}_\alpha \in
\overline{\mathcal{V}}_\alpha = \mathcal{V}_\alpha $ for all $\alpha
\in B.$ Now $\varphi $ is also a homomorphism of $\overline{R}$ onto
$\overline{S}$ and $\overline{R}$ is a subsemigroup of
$\prod\limits_{\alpha \in B} \overline{S}_\alpha \in
\bigvee\limits_{\alpha \in A} \mathcal{V}_\alpha = \mathcal{V}.$ Thus
$\overline{S} \in \mathcal{V}$ \vspace*{-0.2cm}so that $\overline{\mathcal{V}}
\subseteq \mathcal{V}$ and $\mathcal{V}$ is self-dual.\\[0.4cm]

\noindent
(ii) We have
$$
S \in \mathcal{V}^{T_{\ell}} \;\Longleftrightarrow\; S/ \mathcal{L}^0
\in \mathcal{V} \;\Longleftrightarrow\; \overline{S}/ \mathcal{R}^0
\in \overline{\mathcal{V}} = \mathcal{V} \;\Longleftrightarrow
\overline{S} \in \mathcal{V}^{T_r}\,.
$$
Hence $\mathcal{V}^{T_r} = \overline{\mathcal{V}^{T_{\ell}}}.$ A
similar argument yields $\mathcal{V}^{K_r} =
\overline{\mathcal{V}^{K_{\ell}}}.$ \\[0.4cm]

\noindent
(iii) \,\emph{Case\/}: $\mathcal{V}^K$.  \,By Lemma 2.5, $\mathcal{V}^K = [xu_\alpha y
(xv_\alpha y)^{-1} \in E]_{\alpha \in A}.$ Since $\mathcal{V}$ is
self-dual, $\mathcal{V}$ also satisfies $\overline{u}_\alpha =
\overline{v}_\alpha ,$ for all $\alpha \in A.$ Hence, by Lemma 2.5,
$\mathcal{V}^K$ satisfies the identity
$$
y\overline{u}_\alpha x (y\overline{v}_\alpha x)^{-1} \in E \qquad\quad
(\alpha \in A)\,.
$$
By [PR99] Lemma II.2.2(iii), $\mathcal{V}^K$ also satisfies
$$
(y\overline{v}_\alpha x)^{-1} y\overline{u}_\alpha x \in E\,,
$$
that is,
$$
\overline{xu_\alpha y (xv_\alpha y)^{-1}} \in E\,.
$$

\noindent
Thus, by Lemma 3.4(v)\vspace*{0.1cm},
$$
\mathcal{V}^K \subseteq \big[\,\overline{xu_\alpha y (xv_\alpha y)^{-1}}
\in E\big]_{\alpha \in A} = \overline{\mathcal{V}^K} \vspace*{0.1cm}\,
$$
which implies, by Lemma 3.4(v), that $\mathcal{V}^K = \overline{\mathcal{V}^K}$ and
$\mathcal{V}^K$ is self-dual. \\[-0.05cm]

\noindent
\emph{Case\/}: \,$\mathcal{V}^T.$ By Lemma 2.5, $\mathcal{V}^T =
{[u_\alpha^0 = v_\alpha^0, \,(xu_\alpha y)^0 = (xv_\alpha y)^0]}_{\alpha
\in A}.$ \vspace*{0.05cm}Since $\mathcal{V}$ is self-dual, $\mathcal{V}$
also satisfies $\overline{u}_\alpha = \overline{v}_\alpha $ so that
$\mathcal{V}^T$ satisfies $(y\overline{u}_\alpha x)^0 = (y\overline{v}_\alpha
x)^0$ and, \vspace*{0.05cm}equivalently, $\overline{(xu_\alpha y)^0} =
\overline{(xv_\alpha y)^0},$ for all $\alpha \in A.$ In a similar
fashion, \vspace*{0.05cm}it follows that $\mathcal{V}^T$ satisfies
$\overline{u}_\alpha^{\,0} = \overline{v}_\alpha^{\,0}.$ Thus $\mathcal{V}^T
\subseteq \overline{\mathcal{V}^T}$ and equality prevails.\\[0.4cm]
\end{proof}

We may now augment our list of self-dual varieties with $\mathcal{O}
= \mathcal{G}^K,$ $L\mathcal{O} = \mathcal{CS}^K,$ $\mathcal{BG} =
\mathcal{B}^T,$ $\mathcal{B}^{T_{\ell}} \vee
\mathcal{B}^{T_r},$ etc.
\vskip1cm

\section{  Networks}

Before focussing on $K$-classes, we first make some general observations regarding sublattices 
of $\mathcal{L}(\mathcal{CR})$ generated by the repeated application of operators of the 
form $\mathcal{V} \rightarrow \mathcal{V}^{\sigma}, \mathcal{V} \rightarrow 
\mathcal{V}^{\sigma (K_{\ell},K_r)}$ for $\sigma \in \Theta^1$. 

\begin{theorem}
Let $\mathcal{V}
= [u_\alpha =v_\alpha ]_{\alpha \in A} \in [\mathcal{S}, \mathcal{CR}]$
be self-dual.  In addition let $\sigma ,\tau
\in \Theta$ be such that $|\sigma | = |\tau |$ and $h(\sigma ) \neq
h(\tau ),$ $t(\sigma ) = T_{\ell},$ $t(\tau ) = T_r\vspace*{0.1cm}.$
\begin{enumerate}
\item[{\rm (i)}] $\mathcal{V}^{\sigma (K_{\ell},K_r)} \cap
\mathcal{V}^{\tau (K_{\ell},K_r)}$ and\, $\mathcal{V}^{\sigma (T_{\ell},
T_r)} \cap \mathcal{V}^{\tau (T_{\ell},T_r)}$ are self-dual\vspace*{0.1cm}.
\item[{\rm (ii)}] $\mathcal{V}^{\sigma (K_{\ell},K_r)} \vee
\mathcal{V}^{\tau (K_{\ell},K_r)} = \mathcal{V}^{\sigma
(K_{\ell},K_r)K_r} \cap \mathcal{V}^{\tau (K_{\ell},K_r)K_{\ell}}$ is self-dual.
\vspace*{0.1cm}.
\item[{\rm (iii)}] $(\mathcal{V}^\sigma \vee \mathcal{V}^\tau ) \cap
\mathcal{V}^K = (\mathcal{V}^{\sigma T_r} \cap \mathcal{V}^{\tau
T_{\ell}}) \cap \mathcal{V}^K = \mathcal{V}^{\sigma (K_{\ell},K_r)K_r}
\cap \mathcal{V}^{\tau (K_{\ell},K_r)K_{\ell}}\vspace*{0.1cm}.$
\item[{\rm (iv)}] $(\mathcal{V}^\sigma \vee \mathcal{V}^\tau )^{T}
= \mathcal{V}^{\sigma T_r} \cap \mathcal{V}^{\tau T_{\ell}}\vspace*{0.1cm}.$
\item[{\rm (v)}] Let $\mathcal{V} \subset \mathcal{V}^K$. The varieties
of the form $\mathcal{V}^{\rho (K_{\ell},K_r)}$ $(\rho \in \Theta^1)$
and $\mathcal{V}^{\sigma (K_{\ell},K_r)} \cap \mathcal{V}^{\tau
(K_{\ell},K_r)}$ {\rm (}with $\sigma ,\tau $ as above{\rm )} consitute
a sublattice of $[\mathcal{V}, \mathcal{V}^K]$ with distinct elements
as illustrated in Diagram {\rm 4.2}. Solid lines indicate $K_{\ell}$-related
varieties and broken lines indicate $K_r$-related varieties\vspace*{0.1cm}.
\item[{\rm(vi)}] Let $\mathcal{V} \subset \mathcal{V}^K.$ The varieties
of the forms $\mathcal{V}^\rho $ $(\rho \in \Theta^1),$
$\mathcal{V}^\sigma \cap \mathcal{V}^\tau $ and $\mathcal{V}^\sigma
\vee \mathcal{V}^\tau $ {\rm (}with $\sigma ,\tau $ as above{\rm )}
constitute a sublattice of $[\mathcal{V}, \mathcal{CR}]$ with distinct
elements as illustrated in Diagram {\rm 4.3} {\rm (}with the possible
exception of the equality of $\mathcal{V}^\sigma \vee \mathcal{V}^\tau $
and $\mathcal{V}^{\sigma T_r} \cap \mathcal{V}^{\tau T_{\ell}}$ for some
or all of the $\sigma ,\tau ).$ Solid lines indicate
$T_{\ell}$-related varieties and broken lines indicate $T_r$-related
varieties\vspace*{0.1cm}.
\item[{\rm (vii)}] $\bigvee\limits_{\sigma \in \Theta }
\mathcal{V}^{\sigma (T_{\ell},T_r)} = \mathcal{CR}.$
\end{enumerate}
\end{theorem}
\vskip0.025cm

\begin{proof}
(i) \,\emph{Claim\/}: \,$\mathcal{V}^\sigma \cap \mathcal{V}^\tau $
is self-dual. We argue by induction on $|\sigma |.$ Let $|\sigma |
= |\tau | = 1$ so that $\sigma = t(\sigma ) = T_{\ell},$ $\tau = t(\tau ) = T_r.$
Then
$$
\mathcal{V}^{T_{\ell}} \cap \mathcal{V}^{T_r} = \mathcal{V}^{T_{\ell}}
\cap \overline{\mathcal{V}^{T_{\ell}}} \qquad\quad \mbox{by Theorem 3.5 (ii)}
$$
which is self-dual by Lemma 3.4(v).

So now assume that $|\sigma | = |\tau | = m>1$ and that the claim
holds \vspace*{0.2cm}for $|\sigma | = |\tau | < m.$

We may assume that $\sigma = \sigma_1T_{\ell}$ and $\tau = \tau_1T_r$
for some $\sigma_1, \tau_1 \in \Theta $ with $t(\sigma_1) = T_r, t(\tau_1) 
= T_{\ell}$. Then, by Lemma 2.7(iv), we have
$$
(\mathcal{V}^{\sigma_1} \cap \mathcal{V}^{\tau_1})^{T_{\ell}} =
\mathcal{V}^\sigma \cap \mathcal{V}^{\tau_1} = \mathcal{V}^{\tau_1}
\qquad\quad \mbox{since $|\tau_1| < |\sigma |$}
$$
and similarly
$$
(\mathcal{V}^{\sigma_1} \cap \mathcal{V}^{\tau_1})^{T_r} =
\mathcal{V}^{\sigma_1} \cap \mathcal{V}^\tau = \mathcal{V}^{\sigma_1}
\qquad\quad \mbox{since $|\sigma_1| < |\tau |$.}
$$
\vskip0.1cm

\noindent
By Theorem 3.5(ii), since $\mathcal{V}^{\sigma_1} \cap \mathcal{V}^{\tau_1}$
is self-dual by the induction hypothesis, we have
$$
\mathcal{V}^{\sigma_1} = (\mathcal{V}^{\sigma_1} \cap
\mathcal{V}^{\tau_1})^{T_r} = \overline{(\mathcal{V}^{\sigma_1}
\cap \mathcal{V}^{\tau_1})^{T_{\ell}}} =
\overline{\mathcal{V}^{\tau_1}}
$$
so that
$$
\mathcal{V}^{\sigma_1} \vee \mathcal{V}^{\tau_1} =
\overline{\mathcal{V}^{\tau_1}} \vee \mathcal{V}^{\tau_1}\,.
$$
\vskip0.1cm

\noindent
By Lemma 3.4(v),  $\mathcal{V}^{\sigma_1} \vee \mathcal{V}^{\tau_1}$
is self-dual. \\

\noindent
Furthermore, by Lemma 2.7(iii) and with $\tau_1 = \tau_2
T_{\ell},$ where $\tau_2 \in \Theta^1$, we have
$$
(\mathcal{V}^{\sigma_1} \vee \mathcal{V}^{\tau_1})_{T_{\ell}}
= (\mathcal{V}^{\sigma_1})_{T_{\ell}} \vee
(\mathcal{V}^{\tau_1})_{T_{\ell}} = \mathcal{V}^{\sigma_1}
\vee (\mathcal{V}^{\tau_2})_{T_{\ell}}\,.
$$
Since $|\tau_2| < |\sigma_1|,$ it follows that $\mathcal{V}^{\tau_2}
\subseteq \mathcal{V}^{\sigma_1}.$ Consequently,
$$
(\mathcal{V}^{\sigma_1} \vee \mathcal{V}^{\tau_1})_{T_{\ell}}
= \mathcal{V}^{\sigma_1}
$$
which implies that $(\mathcal{V}^{\sigma_1} \vee
\mathcal{V}^{\tau_1})^{T_{\ell}} = \mathcal{V}^{\sigma_1T_{\ell}}
= \mathcal{V}^\sigma \,.$ Similarly, $(\mathcal{V}^{\sigma_1}
\vee \mathcal{V}^{\tau_1})^{T_r} = \mathcal{V}^\tau .$ But
$\mathcal{V}^{\sigma_1} \vee \mathcal{V}^{\tau_1}$ is self-dual and
so, by Theorem 3.5(ii),
$$
\mathcal{V}^\tau = (\mathcal{V}^{\sigma_1} \vee
\mathcal{V}^{\tau_1})^{T_r} = \overline{(\mathcal{V}^{\sigma_1}
\vee \mathcal{V}^{\tau_1})^{T_{\ell}}} =
\overline{\mathcal{V}^\sigma }\,.
$$
Consequently $\mathcal{V}^\sigma \cap \mathcal{V}^\tau $ is
self-dual.\\[-0.15cm]

\noindent
\emph{Claim\/}: \,$\mathcal{V}^{\sigma (K_{\ell}, K_r)} \cap
\mathcal{V}^{\tau (K_{\ell},K_r)}$ is self-dual.
By Lemma 2.6(ii),
\begin{align*}
\mathcal{V}^{\sigma (K_{\ell},K_r)} \cap \mathcal{V}^{\tau
(K_{\ell},K_r)} =\;& \mathcal{V}^{\sigma (T_{\ell},T_r)} \cap
\mathcal{V}^K \cap \mathcal{V}^{\tau (T_{\ell},T_r)} \cap
\mathcal{V}^K \\[0.2cm]
=\;& \big(\mathcal{V}^{\sigma (T_{\ell},T_r)} \cap \mathcal{V}^{\tau
(T_{\ell},T_r)}\big) \cap \mathcal{V}^K
\end{align*}

\noindent
where, by the first case $\mathcal{V}^{\sigma (T_{\ell},T_r)} \cap
\mathcal{V}^{\tau (T_{\ell},T_r)}$ is self-dual and, by Theoem 3.5(iii),
$\mathcal{V}^K$ is self-dual. Hence $\mathcal{V}^{\sigma
(K_{\ell},K_r)} \cap \mathcal{V}^{\tau (K_{\ell},K_r)}$ is
self-dual.\\[0.4cm]

\noindent
(ii) We \vspace*{0.1cm}have
$$
\mathcal{V}^{\sigma (K_{\ell},K_r)} \subseteq \mathcal{V}^{\sigma
(K_{\ell},K_r)} \vee \mathcal{V}^{\tau (K_{\ell},K_r)} \subseteq
\mathcal{V}^{\sigma (K_{\ell},K_r)K_r} \vspace*{0.1cm}\,
$$
so \vspace*{0.1cm}that
$$
\mathcal{V}^{\sigma (K_{\ell},K_r)} \;\;\,K_r \;\;\,
\mathcal{V}^{\sigma (K_{\ell},K_r)} \vee \mathcal{V}^{\tau
(K_{\ell},K_r)} \vspace*{0.1cm}\,
$$
similarly \vspace*{0.1cm}\,
$$
\mathcal{V}^{\tau (K_{\ell},K_r)} \;\;\, K_{\ell} \;\;\,
\mathcal{V}^{\sigma (K_{\ell},K_r)} \vee \mathcal{V}^{\tau
(K_{\ell},K_r)}\vspace*{0.1cm}\,.
$$
On the other \vspace*{0.1cm}hand
$$
\mathcal{V}^{\sigma (K_{\ell},K_r)} \subseteq \mathcal{V}^{\sigma
(K_{\ell},K_r)K_r} \cap \mathcal{V}^{\tau (K_{\ell},K_r)K_{\ell}}
\subseteq \mathcal{V}^{\sigma (K_{\ell},K_r)K_r} \vspace*{0.1cm}\,
$$
so \vspace*{0.1cm}that
$$
\mathcal{V}^{\sigma (K_{\ell},K_r)} \;\;\,K_r \;\;\,
\mathcal{V}^{\sigma (K_{\ell},K_r)K_r} \cap \mathcal{V}^{\tau
(K_{\ell},K_r)K_{\ell}} \vspace*{0.1cm}\,
$$
and \vspace*{0.1cm}similarly
$$
\mathcal{V}^{\tau (K_{\ell},K_r)} \;\;\, K_{\ell} \;\;\,
\mathcal{V}^{\sigma (K_{\ell},K_r)K_r} \cap \mathcal{V}^{\tau
(K_{\ell},K_r)K_{\ell}}\vspace*{0.1cm}\,.
$$
Thus we must have\vspace*{0.1cm}\,
$$
\mathcal{V}^{\sigma (K_{\ell},K_r)} \vee \mathcal{V}^{\tau (K_{\ell},K_r)}
\;\;\, K_{\ell} \cap K_r \;\;\, \mathcal{V}^{\sigma (K_{\ell},K_r)K_r}
\cap \mathcal{V}^{\tau (K_{\ell},K_r)K_{\ell}}\vspace*{0.1cm}\,.
$$
\vskip0.1cm

\noindent
Since $K_{\ell} \cap K_r = \varepsilon$,  we have 
$$
\mathcal{V}^{\sigma (K_{\ell},K_r)} \vee \mathcal{V}^{\tau (K_{\ell},K_r)}
= \mathcal{V}^{\sigma (K_{\ell},K_r)K_r} \cap \mathcal{V}^{\tau
(K_{\ell},K_r)K_{\ell}}\,.
$$
That this variety is self-dual follows from part (i) with 
$\sigma(K_{\ell}, K_r)K_r$ and $\tau(K_{\ell}, K_r)K_{\ell}$ in place of 
$\sigma(K_{\ell}, K_r)$ and $\tau(K_{\ell}, K_r)$, respectively.  
\vskip0.7cm

\noindent
(iii) Denote the three expressions by $A,B$ and $C,$ respectively.
\\[-0.2cm]

\noindent
$A \subseteq B.$ \,We have $\mathcal{V}^{\sigma }, \mathcal{V}^\tau
\subseteq \mathcal{V}^{\sigma T_r} \cap \mathcal{V}^{\tau T_{\ell}}$
so that
$$
(\mathcal{V}^\sigma \vee \mathcal{V}^\tau ) \cap \mathcal{V}^K
\subseteq (\mathcal{V}^{\sigma T_r} \cap \mathcal{V}^{\tau
T_{\ell}}) \cap \mathcal{V}^K\,.
$$
\vskip0.2cm

\noindent
$B=C.$ \,We have
\begin{align*}
(\mathcal{V}^{\sigma T_r} \cap \mathcal{V}^{\tau T_{\ell}})
\cap \mathcal{V}^K =\;& (\mathcal{V}^{\sigma T_r} \cap
\mathcal{V}^{K}) \cap (\mathcal{V}^{\tau T_{\ell}} \cap
\mathcal{V}^K) \\[0.2cm]
=\;& \mathcal{V}^{\sigma (K_{\ell},K_r)K_r} \cap \mathcal{V}^{\tau
(K_{\ell},K_r)K_{\ell}} \qquad\quad\mbox{by Lemma 2.6}
\end{align*}
\vskip0.2cm

\noindent
$C\subseteq A.$ \,By part (ii),
\begin{align*}
\mathcal{V}^{\sigma (K_{\ell},K_r)K_r} \cap \mathcal{V}^{\tau
(K_{\ell},K_r)K_{\ell}} =\;& \mathcal{V}^{\sigma (K_{\ell},K_r)} \vee
\mathcal{V}^{\tau (K_{\ell},K_r)} \\[0.2cm]
=\;& (\mathcal{V}^{\sigma}\cap \mathcal{V}^{\kappa}) \vee (\mathcal{V}^{\tau}\cap 
\mathcal{V}^{\kappa})\\
\subseteq\;& (\mathcal{V}^\sigma \vee \mathcal{V}^\tau ) \cap
\mathcal{V}^K\,.
\end{align*}
\vskip0.85cm

\noindent
(iv) Replacing $K_{\ell}$ and $K_r$ by $T_{\ell}$ and $T_r,$
respectively, in the argument in part (ii) we \vspace*{0.1cm}obtain
$$
\mathcal{V}^\sigma \vee \mathcal{V}^\tau  \;\;\, P \;\;\,
\mathcal{V}^{\sigma T_r} \cap \mathcal{V}^{\tau T_{\ell}}
\vspace*{0.1cm}\,
$$
for $P = T_{\ell}$ and $P = T_r$ and therefore also for $T_{\ell}
 \cap T_r = T.$ Thus $\mathcal{V}^\sigma \vee \mathcal{V}^\tau \;\;
T \;\; \mathcal{V}^{\sigma T_r} \cap \mathcal{V}^{\tau T_{\ell}}.$
Now $T$ is a complete congruence on $\mathcal{L}(\mathcal{CR})$ while
$\mathcal{V}^{\sigma T_r} = \mathcal{V}^{\sigma T_rT}$ and
$\mathcal{V}^{\tau T_{\ell}} = \mathcal{V}^{\tau T_{\ell}T}$ are both
maximal in their $T$-class. Hence so also is $\mathcal{V}^{\sigma
T_r} \cap \mathcal{V}^{\tau T_{\ell}},$ that is,
$$
\mathcal{V}^{\sigma T_r} \cap \mathcal{V}^{\tau T_{\ell}} =
(\mathcal{V}^{\sigma } \vee \mathcal{V}^\tau )^T\,.
$$
\vskip0.85cm

\noindent
(v) Since $\mathcal{V} \subset \mathcal{V}^K$ and, by Lemma 2.8(ii),
$\mathcal{V}^K = \bigvee\limits_{\sigma \in\Theta^1}
\mathcal{V}^{\sigma (K_{\ell},K_r)},$ it follows that either
$\mathcal{V} \neq \mathcal{V}^{K_{\ell}}$ or $\mathcal{V} \neq
\mathcal{V}^{K_r}.$ By hypothesis, $\mathcal{V}$ is self-dual. Hence
$\mathcal{V} \neq \mathcal{V}^{K_{\ell}}, \mathcal{V}^{K_r}.$ We also
have $\mathcal{V} = \mathcal{V}^{K_{\ell}} \cap \mathcal{V}^{K_r}$
so that $\mathcal{V}^{K_{\ell}}$ and $\mathcal{V}^{K_r}$ must be
incomparable.  Thus we have the base $\mathcal{V}, \mathcal{V}^{K_{\ell}}, \mathcal{V}^{K_r}$ 
for Diagram 4.2.   This, in turn, implies that $\mathcal{V}^{K_{\ell}}
\vee \mathcal{V}^{K_r} \neq \mathcal{V}^{K_{\ell}}, \mathcal{V}^{K_r}.$
In addition $\mathcal{V}^{K_r} = \overline{\mathcal{V}^{K_{\ell}}}$
so that $\mathcal{V}^{K_{\ell}} \vee \mathcal{V}^{K_r}$ is self-dual.

Suppose that $\mathcal{V}^{K_{\ell}} \vee \mathcal{V}^{K_r} =
\mathcal{V}^K.$ \vspace*{0.1cm}Then
$$
\mathcal{V}^K = \mathcal{V}^{KK_{\ell}} = (\mathcal{V}^{K_{\ell}}
\vee \mathcal{V}^{K_r})^{K_{\ell}} \subseteq
\mathcal{V}^{K_rK_{\ell}} \subseteq \mathcal{V}^K \vspace*{0.1cm}\,
$$
so that $\mathcal{V}^K = \mathcal{V}^{K_rK_{\ell}}.$ Hence, either
$\mathcal{V} \subset \mathcal{V}^K$ and $\mathcal{V}^{K_r} = \mathcal{V}^K,$
or $\mathcal{V}^{K_r} \subset \mathcal{V}^K$ and $\mathcal{V}^{K_rK_{\ell}}
= \mathcal{V}^K.$ This implies that either
$$
(\mathcal{V}^K)_{T_r} = (\mathcal{V}^{K_r})_{T_r} =
\mathcal{V}_{T_r} \subseteq \mathcal{V} \subset \mathcal{V}^K
$$
or \vspace*{0.1cm}\,
$$
(\mathcal{V}^K)_{T_{\ell}} =
(\mathcal{V}^{K_rK_{\ell}})_{T_{\ell}} =
(\mathcal{V}^{K_r})_{T_{\ell}} \subseteq \mathcal{V}^{K_r}
\subset \mathcal{V}^K
$$
\vskip0.25cm

\noindent
contradicting either Lemma 2.6(i) or its dual. Hence $\mathcal{V}^{K_{\ell}}
\vee \mathcal{V}^{K_r}$ is a self-dual variety that is a proper
subvariety of $(\mathcal{V}^{K_{\ell}} \vee
\mathcal{V}^{K_r})^K = \mathcal{V}^K.$ In addition, by part
\vspace*{0.1cm}(ii), with $\sigma = T_{\ell}, \tau = T_r$ and Lemma 2.7 (v),
\begin{align*}
(\mathcal{V}^{K_{\ell}} \vee \mathcal{V}^{K_r})^{K_{\ell}}
=\;& (\mathcal{V}^{K_{\ell}K_r} \cap
\mathcal{V}^{K_rK_{\ell}})^{K_{\ell}} =
\mathcal{V}^{K_{\ell}K_rK_{\ell}} \cap \mathcal{V}^{K_rK_{\ell}}
\\[0.2cm]
=\;& \mathcal{V}^{K_rK_{\ell}}
\end{align*}
\vskip0.1cm

\noindent
and similarly $(\mathcal{V}^{K_{\ell}} \vee
\mathcal{V}^{K_r})^{K_r} = \mathcal{V}^{K_{\ell}K_r}.$ Thus we
can build the next step of the diagram starting with
$\mathcal{V}^{K_{\ell}} \vee \mathcal{V}^{K_r}$ and so on.\\[0.4cm]

\noindent
(vi) Note that $\mathcal{V}^\rho \subseteq \mathcal{V}^\sigma $ for
all $\rho ,\sigma \in \Theta $ with $|\rho | < |\sigma |$ so that
with the help of part (iv) we see that the partially ordered set
of the varieties listed in this part is as depicted in Diagram 4.3.
As in part (v), we deduce that
$\mathcal{V} \neq \mathcal{V}^{K_{\ell}}$ and $\mathcal{V} \neq
\mathcal{V}^{K_r}.$  Again, 
$\mathcal{V} = \mathcal{V}^{K_{\ell}} \cap \mathcal{V}^{K_r}$
so that the three varieties $\mathcal{V}, \mathcal{V}^{K_{\ell}},
\mathcal{V}^{K_r}$ are all distinct. 

We have $\mathcal{V} \subseteq \mathcal{V}^{K_{\ell}}\cap \mathcal{V}^T \subseteq 
\mathcal{V}^K \cap \mathcal{V}^T = \mathcal{V}$.  Therefore $\mathcal{V}^{K_{\ell}}$ and 
$\mathcal{V}^T$ ae incomparable.  Hence $\mathcal{V}^{T_{\ell}} \neq \mathcal{V}^T$ and, 
dually, $\mathcal{V}^{T_r} \neq \mathcal{V}^{T}$.   
But $\mathcal{V}^T = \mathcal{V}^{T_{\ell}} \cap \mathcal{V}^{T_r}$.  It
follows that $\mathcal{V}^{T_{\ell}}$ and $\mathcal{V}^{T_r}$ are
incomparable and therefore that $\mathcal{V}^{T_{\ell}} \vee
\mathcal{V}^{T_r}$ is distinct from $\mathcal{V}^{T_{\ell}}$ and
$\mathcal{V}^{T_r}.$ It is possible that $\mathcal{V}^{T_{\ell}} \vee
\mathcal{V}^{T_r} = \mathcal{V}^{T_{\ell}T_r} \cap
\mathcal{V}^{T_rT_{\ell}}$ but definitely this is not always the case, see example
below. However, by Theorem 3.5(ii) and Lemma 3.4(v), $\mathcal{V}^{T_{\ell}}
\vee \mathcal{V}^{T_r}$ is self-dual and therefore, by Theorem 3.5(iii),
$\mathcal{V}^{T_{\ell}T_r} \cap \mathcal{V}^{T_rT_{\ell}} =
(\mathcal{V}^{T_{\ell}} \vee \mathcal{V}^{T_r})^{T}$ is
self-dual.

Suppose that $\mathcal{V}^{T_{\ell}T_r} \cap \mathcal{V}^{T_rT_{\ell}}
= \mathcal{CR}.$ Since $\mathcal{V}^{T_{\ell}T_r} \subseteq
\mathcal{CR},$ we must have $\mathcal{V}^{T_{\ell}T_r} = \mathcal{CR}.$
This means that either $\mathcal{V}^{T_{\ell}} < \mathcal{CR}$ and
$(\mathcal{V}^{T_{\ell}})^{T_r} = \mathcal{CR}$ or $\mathcal{V}
< \mathcal{CR}$ and $\mathcal{V}^{T_{\ell}} = \mathcal{CR}.$ In the
former case $\mathcal{CR}_{T_r} \subseteq \mathcal{V}^{T_{\ell}} <
\mathcal{CR}$ and in the latter case $\mathcal{CR}_{T_{\ell}}
\subseteq \mathcal{V} < \mathcal{CR}.$ Since $\mathcal{CR} =
\mathcal{CR}^K,$ this contradicts Lemma 2.6(i) or its dual. Thus
$\mathcal{V}^{T_{\ell}T_r} \cap \mathcal{V}^{T_rT_{\ell}} =
(\mathcal{V}^{T_{\ell}} \vee \mathcal{V}^{T_r})^T$ is
self-dual and a proper subvariety of $\mathcal{CR}.$

Furthermore,
$$
\begin{array}{rcll}
((\mathcal{V}^{T_{\ell}} \vee \mathcal{V}^{T_r})^T)^{T_{\ell}}
&\!\!\!=\!\!\!& (\mathcal{V}^{T_{\ell}T_r} \cap
\mathcal{V}^{T_rT_{\ell}})^{T_{\ell}} & \quad \mbox{by part (iv)}
\\[0.3cm]
&\!\!\!=\!\!\!& \mathcal{V}^{T_{\ell}T_rT_{\ell}} \cap
\mathcal{V}^{T_rT_{\ell}} & \quad \mbox{by Theorem 2.7(iv)} \\[0.3cm]
&\!\!\!=\!\!\!& (\mathcal{V}^{T_r})^{T_{\ell}}
\end{array}
$$
\vskip0.2cm

\noindent
and dually $((\mathcal{V}^{T_{\ell}} \vee
\mathcal{V}^{T_r})^T)^{T_r} =
(\mathcal{V}^{T_{\ell}})^{T_r}.$

We may now repeat the above argument starting with
$(\mathcal{V}^{T_{\ell}} \vee \mathcal{V}^{T_r})^T$ to extend
Diagram 4.3 to the next level, and so on. \\[0.4cm]

\noindent
(vii) See [PR88] Theorem 4.6(iii).
\end{proof}

\newpage

\[ \hspace*{2cm}\begin{minipage}[t]{12cm}
\beginpicture
\setcoordinatesystem units <.7truecm,.7truecm>
\setplotarea x from 0 to 14, y from -1 to 10
\setlinear
\setdashes <1mm>
\plot 6 0  8.5 1.25 /
\plot 3.5 1.25  6 2.5 /
\plot 6 2.5  8.5 3.75 /
\plot 3.5 3.75  6 5 /
\plot 8.5 6.25  6 5 /
\plot 3.5 6.25  6 7.5 /
\plot 6 7.5  7.5 8.225 /

\setsolid
\plot 3.5 1.25  6 0 /
\plot 6 2.5  8.5 1.25 /
\plot 3.5 3.75  6 2.5 /
\plot 8.5 3.75  6 5 /
\plot 3.5 6.25  6 5 /
\plot 6 7.5  8.5 6.25 /
\plot 6 7.5  4.5 8.225 /

\setsolid
\put {\circle*{2.5}} at 6.075 0
\put {\circle*{2.5}} at 3.575 1.25
\put {\circle*{2.5}} at 8.575 1.25
\put {\circle*{2.5}} at 6.075 2.5
\put {\circle*{2.5}} at 3.575 3.75
\put {\circle*{2.5}} at 8.575 3.75
\put {\circle*{2.5}} at 6.075 5
\put {\circle*{2.5}} at 3.575 6.25
\put {\circle*{2.5}} at 8.575 6.25
\put {\circle*{2.5}} at 6.075 7.5

\put {\circle*{2}} at 6.05 8.4
\put {\circle*{2}} at 6.05 8.6
\put {\circle*{2}} at 6.05 8.8

\put {${\mathcal{V}}$} at 6 -0.4
\put {${\mathcal{V}}^{K_{\ell}}$} at  2.9 1.25
\put {${\mathcal{V}}^{K_rK_{\ell}}$} at 2.6 3.75
\put {${\mathcal{V}}^{K_{\ell}K_rK_{\ell}}$} at 2.35 6.25
\put {${\mathcal{V}}^{K_r}$} at 9.175 1.25

\put {${\mathcal{V}}^{K_{\ell}} \vee {\mathcal{V}}^{K_r} =
{\mathcal{V}}^{K_rK_{\ell}} \cap {\mathcal{V}}^{K_{\ell}K_r}$} at
11.5 2.5

\put {${\mathcal{V}}^{K_{\ell}K_r}$} at 9.45 3.75

\put {${\mathcal{V}}^{K_rK_{\ell}} \vee {\mathcal{V}}^{K_{\ell}K_r}
= {\mathcal{V}}^{K_{\ell}K_rK_{\ell}} \cap
{\mathcal{V}}^{K_rK_{\ell}K_r}$} at 12.625 5

\put {${\mathcal{V}}^{K_rK_{\ell}K_r}$} at 9.75 6.25

\put {${\mathcal{V}}^K$} at 6.2 9.3
\endpicture
\end{minipage} \]
\vskip0.3cm

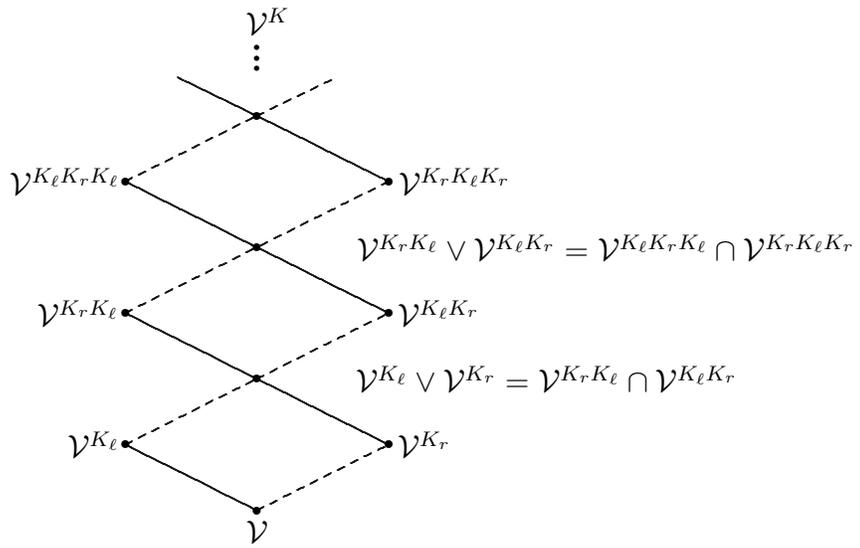
\centerline{\bf Diagram 4.2 \hspace*{1.5cm}}

\vskip1.4cm

\[ \hspace*{-0.35cm}\begin{minipage}[t]{12cm}
\beginpicture
\setcoordinatesystem units <.7truecm,.7truecm>
\setplotarea x from 0 to 14, y from -1 to 12
\setlinear
\setdashes <1mm>
\plot 6 0  8.5 1.25 /
\plot 3.5 1.25  6 2.5 /

\plot 6 3.7  8.5 4.95 /
\plot 3.5 4.95  6 6.2 /

\plot 8.5 8.65  6 7.4 /
\plot 3.5 8.65  6 9.9 /

\plot 6.035 2.5  6.035 3.7 /
\plot 6.035 6.2  6.035 7.4 /

\plot 6.035 9.9  6.035 10.6 /

\setsolid
\plot 5.975 0  5.975 -1 /

\plot 3.5 1.25  6 0 /
\plot 6 2.5  8.5 1.25 /

\plot 3.5 4.95  6 3.7 /
\plot 8.5 4.95  6 6.2 /

\plot 3.5 8.65  6 7.4 /
\plot 6 9.9  8.5 8.65 /

\plot 5.95 2.5  5.95 3.7 /
\plot 5.95 6.2  5.95 7.4 /

\plot 5.95 9.9  5.95 10.6 /

\setsolid
\put {\circle*{2.5}} at 6.05 0
\put {\circle*{2.5}} at 6.05 -1
\put {\circle*{2.5}} at 3.575 1.25
\put {\circle*{2.5}} at 8.575 1.25
\put {\circle*{2.5}} at 6.075 2.5

\put {\circle*{2.5}} at 6.075 3.7
\put {\circle*{2.5}} at 3.575 4.95
\put {\circle*{2.5}} at 8.575 4.95
\put {\circle*{2.5}} at 6.075 6.2

\put {\circle*{2.5}} at 6.075 7.4
\put {\circle*{2.5}} at 3.575 8.65
\put {\circle*{2.5}} at 8.575 8.65
\put {\circle*{2.5}} at 6.075 9.9

\put {\circle*{2}} at 6.025 10.9
\put {\circle*{2}} at 6.025 11.1
\put {\circle*{2}} at 6.025 11.3

\put {${\mathcal{V}}$} at 6.4 -1
\put {${\mathcal{V}}^T$} at 6.6 -0.1
\put {${\mathcal{V}}^{T_{\ell}}$} at 2.875 1.25
\put {${\mathcal{V}}^{T_rT_{\ell}}$} at 2.65 4.95
\put {${\mathcal{V}}^{T_{\ell}T_rT_{\ell}}$} at 2.45 8.7
\put {${\mathcal{V}}^{T_r}$} at 9.1 1.25

\put {${\mathcal{V}}^{T_{\ell}} \vee {\mathcal{V}}^{T_r}$} at 9
2.5
\put {${\mathcal{V}}^{T_rT_{\ell}} \cap \mathcal{V}^{T_{\ell}T_r} = (\mathcal{V}^{T_{\ell}} 
\vee \mathcal{V}^{T_r})^T$} at
10.6 3.7

\put {${\mathcal{V}}^{T_{\ell}T_r}$} at 9.325 4.95

\put {${\mathcal{V}}^{T_rT_{\ell}} \vee {\mathcal{V}}^{T_{\ell}T_r}$}
at 9.4 6.2
\put {${\mathcal{V}}^{T_{\ell}T_rT_{\ell}} \cap
{\mathcal{V}}^{T_rT_{\ell}T_r} = (\mathcal{V}^{T_rT_{\ell}}\vee \mathcal{V}^{T_{\ell}T_r})^T$} at 11 7.4

\put {${\mathcal{V}}^{T_rT_{\ell}T_r}$} at 9.55 8.65

\put {${\mathcal{CR}}$} at 6 11.75
\endpicture
\end{minipage} \]
\vskip0.3cm

\centerline{\bf Diagram 4.3 \hspace*{1.5cm}}
\vskip1.6cm

\addtocounter{theorem}{2}
\begin{theorem}
Let $\mathcal{V} = \mathcal{V}^T \in \mathcal{L}(\mathcal{CR}), \mathcal{V} \neq \mathcal{CR}$ 
and $\sigma \in \Theta^1.$
Then $\mathcal{V}^\sigma $ is the largest variety
$\mathcal{W} \in \mathcal{L}(\mathcal{CR})$ such that $\mathcal{W}
\cap \mathcal{V}^K = \mathcal{V}^{\sigma (K_{\ell},K_r)}.$
If $\sigma, \tau \in \Theta $ are such that
$|\sigma | = |\tau |,$ $h(\sigma ) \neq h(\tau ),$ then $\mathcal{V}^\sigma
\cap \mathcal{V}^\tau $ is the  largest variety  $\mathcal{W} \in
\mathcal{L}(\mathcal{CR})$ such that $\mathcal{W} \cap \mathcal{V}^K
= \mathcal{V}^{\sigma (K_{\ell},K_r)} \cap \mathcal{V}^{\tau (K_{\ell},K_r)}.$
\end{theorem}

\begin{proof}
Recall from Lemma 2.6(ii) that, for any $\sigma \in \Theta^1,$ we \vspace*{-0.1cm}have
$$
\mathcal{V}^\sigma \cap \mathcal{V}^K = \mathcal{V}^{\sigma
(K_{\ell},K_r)}. \vspace*{-0.05cm}\,
$$

Now let $\mathcal{W}\in \mathcal{L}(\mathcal{CR})$ be such that $\mathcal{W} \cap
\mathcal{V}^K \subseteq \mathcal{V}^{\sigma (K_{\ell},K_r)}$. For $\sigma =
\emptyset,$ this means 
that $\mathcal{W} \cap \mathcal{V}^K \subseteq \mathcal{V}.$ Suppose
that $\mathcal{W} \not\subseteq \mathcal{V}.$ By Lemma 2.10, 
there exists $S\in \mathcal{W} \backslash
\mathcal{V}$ such that either $S \in \mathcal{V}^{K_{\ell}}\backslash
\mathcal{V}$, or $S\in \mathcal{V}^{K_r}\backslash \mathcal{V}$ 
so that $S\in (\mathcal{W} \cap \mathcal{V}^K) \backslash
\mathcal{V}$ which is a contradiction. Hence $\mathcal{W} \subseteq
\mathcal{V}.$

Now consider $\sigma \in \Theta $ and assume that the claim holds
for shorter words than $\sigma .$ By duality, if suffices to consider
the case where $\sigma = \sigma_1T_{\ell},$ $\sigma_1 \in \Theta^1.$
Let $\mathcal{W} \in \mathcal{L}(\mathcal{CR})$ be such that
$\mathcal{W} \cap \mathcal{V}^K \subseteq \mathcal{V}^{\sigma
(K_{\ell},K_r)}.$ Then, by Lemma 2.7(iii) and the induction hypothesis,
\begin{align*}
\mathcal{W}_{T_{\ell}} \cap \mathcal{V}^K =\;& (\mathcal{W} \cap
\mathcal{V}^K)_{T_{\ell}} \subseteq (\mathcal{V}^{\sigma
(K_{\ell},K_r)})_{T_{\ell}} =
(\mathcal{V}^{\sigma_1
(K_{\ell},K_r)K_{\ell}})_{T_{\ell}} \\[0.15cm]
\subseteq\;& \mathcal{V}^{\sigma_1(K_{\ell},K_r)}\,.
\end{align*}

\noindent
By the induction hypothesis $\mathcal{W}_{T_{\ell}} \subseteq
\mathcal{V}^{\sigma_1}.$ \vspace*{0.05cm}Hence
$$
\mathcal{W} \subseteq (\mathcal{W}_{T_{\ell}})^{T_{\ell}}
\subseteq \mathcal{V}^{\sigma_1T_{\ell}} = \mathcal{V}^\sigma
\vspace*{-0.2cm}
$$
and the induction step holds.

Therefore $\mathcal{V}^\sigma $ must
be the largest \vspace*{0.05cm}variety whose intersection with
$\mathcal{V}^K$ is $\mathcal{V}^{\sigma (K_{\ell},K_r)}.$

For $\sigma ,\tau $ as in the statement and $\mathcal{W} \in
\mathcal{L}(\mathcal{CR}),$ we \vspace*{-0.1cm}have
\begin{eqnarray*}
\mathcal{W} \cap \mathcal{V}^K  \subseteq 
\mathcal{V}^{\sigma(K_{\ell},K_r)} \cap \mathcal{V}^{\tau(K_{\ell},K_r)} 
&\Longrightarrow&  \mathcal{W} \cap \mathcal{V}^K \subseteq \mathcal{V}^{\sigma(K_{\ell}, K_r)}\\
&\Longrightarrow& \mathcal{W} \subseteq \mathcal{V}^{\sigma} \;\mbox{by argument above} .
\end{eqnarray*} 
Similarly $\mathcal{W} \cap \mathcal{V}^K  \subseteq 
\mathcal{V}^{\sigma(K_{\ell},K_r)} \cap \mathcal{V}^{\tau(K_{\ell},K_r)} \Longrightarrow 
\mathcal{W} \subseteq \mathcal{V}^{\tau}$
so that $\mathcal{W} \subseteq  \mathcal{V}^{\sigma}\cap  \mathcal{V}^{\tau}$.
Conversely \vspace*{-0.1cm}\,
\begin{align*}
(\mathcal{V}^\sigma \cap \mathcal{V}^\tau )  \cap
\mathcal{V}^K =\;& (\mathcal{V}^\sigma \cap \mathcal{V}^K)
\cap (\mathcal{V}^\tau \cap \mathcal{V}^K) \\[0.15cm]
=\;& \mathcal{V}^{\sigma (K_{\ell},K_r)} \cap \mathcal{V}^{\tau
(K_{\ell},K_r)}
\end{align*}
and the final claim holds.
\end{proof}

Theorem 4.4 is modelled on the result of Reilly and Zhang [RZ], Lemma 4.11, characterizing 
the largest subvariety of $\mathcal{L}(\mathcal{CR})$ whose intersection with $\mathcal{B}$ 
is a specific variety of bands.\\

We note that, for any variety $\mathcal{V} = [u_\alpha =v_\alpha
]_{\alpha \in A} \in \mathcal{L}(\mathcal{CR}),$ we can derive a basis
of identities for each of the varieties appearing in Diagrams 4.2
and 4.3 (except for those in 4.3 expressed as the join of two
varieties) by a sequence of applications of the bases provided in
Lemma 2.5 for varieties of each of the forms $\mathcal{U}^K,$
$\mathcal{U}^{T_{\ell}},$ $\mathcal{U}^{T_r},$
$\mathcal{U}^{K_{\ell}},$ $\mathcal{U}^{K_r}$ $(\mathcal{U} \in
\mathcal{L}(\mathcal{CR}))$ and taking the union of two bases for
the intersection of varieties.   Diagram 4.5 displays the relationship between 
the network based on $\mathcal{V}$ using the upper operators associated with 
$K_{\ell}, K_r$ and the network based on $\mathcal{V}$ using the upper operators associated with 
$T_{\ell}, T_r$.


\ \\[0.5cm]

\[ \hspace*{1cm}\begin{minipage}[t]{14cm}
\beginpicture
\setcoordinatesystem units <.7truecm,.7truecm>
\setplotarea x from -6.5 to 14, y from 4 to 25
\setplotsymbol ({\eightrm.})
\setlinear
\setlinear
\setdashes <1mm>
\plot 4.2 13  6.7 14.25 /
\plot 1.7 14.25  4.2 15.5 /

\plot 4.2 16.7  6.7 17.95 /
\plot 1.7 17.95  4.2 19.2 /

\plot 6.7 21.65  4.2 20.4 /
\plot 1.7 21.65  4.2 22.9 /

\plot 4.235 15.5  4.235 16.7 /
\plot 4.235 19.2  4.235 20.4 /

\plot 4.235 22.9  4.235 23.6 /

\setsolid
\plot 4.175 13  4.175 12 /

\plot 1.7 14.25  4.2 13 /
\plot 4.2 15.5  6.7 14.25 /

\plot 1.7 17.95  4.2 16.7 /
\plot 6.7 17.95  4.2 19.2 /

\plot 1.7 21.65  4.2 20.4 /
\plot 4.2 22.9  6.7 21.65 /

\plot 4.15 15.5  4.15 16.7 /
\plot 4.15 19.2  4.15 20.4 /

\plot 4.15 22.9  4.15 23.6 /

\setsolid
\put {\circle*{2.5}} at 4.25 13
\put {\circle*{2.5}} at 4.25 12
\put {\circle*{2.5}} at 1.775 14.25
\put {\circle*{2.5}} at 6.775 14.25
\put {\circle*{2.5}} at 4.275 15.5

\put {\circle*{2.5}} at 4.275 16.7
\put {\circle*{2.5}} at 1.775 17.95
\put {\circle*{2.5}} at 6.775 17.95
\put {\circle*{2.5}} at 4.275 19.2

\put {\circle*{2.5}} at 4.275 20.4
\put {\circle*{2.5}} at 1.775 21.65
\put {\circle*{2.5}} at 6.775 21.65
\put {\circle*{2.5}} at 4.275 22.9

\put {\circle*{2}} at 4.225 23.9
\put {\circle*{2}} at 4.225 24.1
\put {\circle*{2}} at 4.225 24.3

\put {${\mathcal{V}}$} at 4.5 12
\put {${\mathcal{V}}^T$} at 4.7 12.85
\put {${\mathcal{V}}^{T_{\ell}}$} at 1.075 14.25
\put {${\mathcal{V}}^{T_rT_{\ell}}$} at 0.85 17.95
\put {${\mathcal{V}}^{T_{\ell}T_rT_{\ell}}$} at 0.65 21.7
\put {${\mathcal{V}}^{T_r}$} at 7.35 14.25

\put {${\mathcal{V}}^{T_{\ell}} \vee {\mathcal{T}}^{T_r}$} at 6.55
15.65
\put {${\mathcal{V}}^{T_rT_{\ell}} \cap {\mathcal{V}}^{T_{\ell}T_r}$} at
6.95 16.55

\put {${\mathcal{V}}^{T_{\ell}T_r}$} at 7.575 17.95

\put {${\mathcal{V}}^{T_rT_{\ell}} \vee {\mathcal{V}}^{T_{\ell}T_r}$}
at 7.05 19.4
\put {${\mathcal{V}}^{T_{\ell}T_rT_{\ell}} \cap
{\mathcal{V}}^{T_rT_{\ell}T_r}$} at 7.5 20.225

\put {${\mathcal{V}}^{T_rT_{\ell}T_r}$} at 7.795 21.65

\put {${\mathcal{CR}}$} at 4.2 24.75

\setdashes <1mm>
\plot -3 4.5  -0.5 5.75 /
\plot -5.5 5.75  -3 7 /
\plot -3 7  -0.5 8.25 /
\plot -5.5 8.25  -3 9.5 /
\plot -0.5 10.75  -3 9.5 /
\plot -5.5 10.75  -3 12 /
\plot -3 12  -1.5 12.725 /

\setsolid
\plot -5.5 5.75  -3 4.5 /
\plot -3 7  -0.5 5.75 /
\plot -5.5 8.25  -0.5  5.75 /
\plot -0.5 8.25  -3 9.5 /
\plot -5.5 10.75 -3 9.5 /

\plot -3 12  -0.5 10.75 /
\plot -3 12  -4.5 12.725 /

\setsolid
\put {\circle*{2.5}} at -2.925 4.5
\put {\circle*{2.5}} at -5.425 5.75
\put {\circle*{2.5}} at -0.425 5.75
\put {\circle*{2.5}} at -2.925 7
\put {\circle*{2.5}} at -5.425 8.25
\put {\circle*{2.5}} at -0.425 8.25
\put {\circle*{2.5}} at -2.925 9.5
\put {\circle*{2.5}} at -5.425 10.75
\put {\circle*{2.5}} at -0.425 10.75
\put {\circle*{2.5}} at -2.925 12

\put {\circle*{2}} at -2.95 12.9
\put {\circle*{2}} at -2.95 13.1
\put {\circle*{2}} at -2.95 13.3

\put {${\mathcal{V}}$} at -3.1 4
\put {${\mathcal{V}}^{K_{\ell}}$} at  -6.2 5.75
\put {${\mathcal{V}}^{K_rK_{\ell}}$} at -6.5 8.25
\put {${\mathcal{V}}^{K_{\ell}K_rK_{\ell}}$} at -6.75 10.75
\put {${\mathcal{V}}^{K_r}$} at 0.35 5.75

\put {${\mathcal{V}}^{K_{\ell}}\!\vee\!{\mathcal{V}}^{K_r}\!=\!
{\mathcal{V}}^{K_rK_{\ell}}\!\cap\!{\mathcal{V}}^{K_{\ell}K_r}$} at
1.7 7

\put {${\mathcal{V}}^{K_{\ell}K_r}$} at 0.65 8.25

\put {${\mathcal{V}}^{K_rK_{\ell}}\!\vee\!{\mathcal{V}}^{K_{\ell}K_r}
\!=\!{\mathcal{V}}^{K_{\ell}K_rK_{\ell}}\!\cap\!
{\mathcal{V}}^{K_rK_{\ell}K_r}$} at 2.825 9.5

\put {${\mathcal{V}}^{K_rK_{\ell}K_r}$} at 0.875 10.75

\put {${\mathcal{V}}^K$} at -2.8 13.8

\setplotsymbol ({\fiverm.})

\plot -5.475 5.75  -4.075 7.4 /
\plot -3.95 7.575  -2.575 9.175 /
\plot -2.45 9.35  -1.925 9.95 /
\plot -1.705 10.225  1.7 14.25 /

\plot -5.475 8.25  -4.225 9.95 /
\plot -4.0575 10.15  -2.825 11.8225 /
\plot -2.75 11.95  -2.605 12.125 /
\plot -2.4 12.4  1.7 17.95 /

\plot -5.475 10.75  1.7 21.65 /

\plot -0.5 5.75  0.4 6.8 /
\plot 0.8 7.275  2.375 9.15 /
\plot 2.95 9.825  6.7 14.25 /

\plot -0.5 8.25  0.175 9.15 /
\plot 0.6275 9.775  1.185 10.525 /
\plot 1.6 11.075  3.2725 13.3275 /
\plot 3.4 13.5  4.65 15.175 /
\plot 4.75 15.325  5.4 16.2 /
\plot 5.905 16.9  6.7 17.95 /

\plot -0.5 10.75  1.7 14.075 /
\plot 1.975 14.5  3.55 16.8755  /
\plot 3.7 17.1  4.8 18.7725 /
\plot 4.95 19  5.5 19.85 /
\plot 6 20.605  6.7 21.65 /

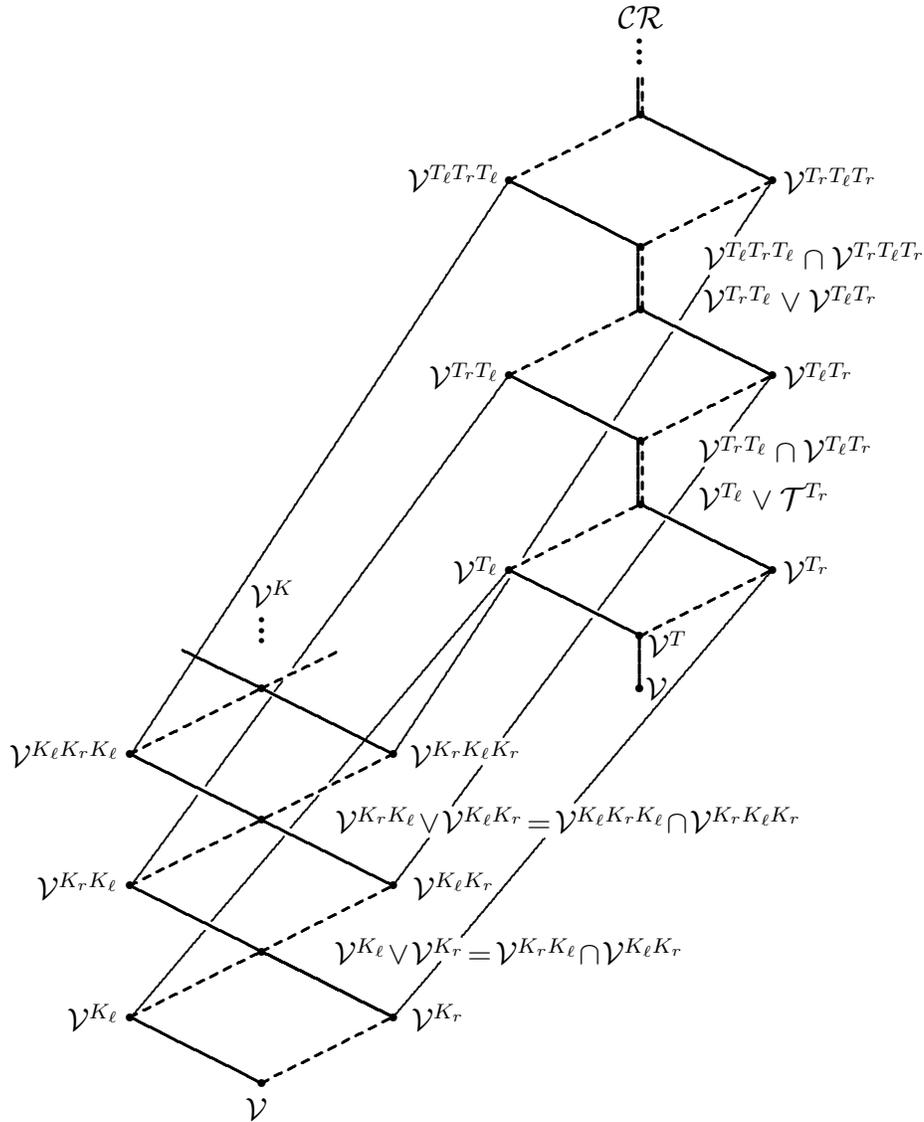
\put{{\bf{Diagram 4.5}}} at 2 3
\endpicture
\end{minipage} \]

\vskip1.6cm


\addtocounter{theorem}{1}
\newpage

\begin{example}
{\rm In general the vertical lines in Diagram 4.5 can represent
non-trivial intervals. For example, if we take $\mathcal{V} =
\mathcal{SG},$ the variety of all semilattices of groups, then
$$
\begin{array}{rcl c l}
\mathcal{V}^{T_{\ell}} &\!\!=\!\!& \mathcal{LRO} & \mbox{---} &
\mbox{the variety of all left regular orthogroups} \\[0.3cm]
\mathcal{V}^{T_r} &\!\!=\!\!& \mathcal{RRO} & \mbox{---} & \mbox{the
variety of all right regular orthogroups} \\[0.3cm]
\mathcal{V}^{T_{\ell}} \vee \mathcal{V}^{T_r} &\!\!=\!\!& \mathcal{RO}
&\mbox{---} & \mbox{the variety of all regular orthogroups} \\[0.3cm]
\mathcal{V}^{T_{\ell}T_r} &\!\!=\!\!& \mathcal{R}^* & \mbox{---} &
\mbox{\parbox[t]{7cm}{the variety of all completely regular semigroups
on which $\mathcal{R}$ is a congruence}} \\[0.8cm]
\mathcal{V}^{T_rT_{\ell}} &\!\!=\!\!& \mathcal{L}^* & \mbox{---} &
\mbox{the dual of $\mathcal{R}^*$} \\[0.3cm]
\mathcal{V}^{T_{\ell}T_r} \cap \mathcal{V}^{T_rT_{\ell}} &\!\!=\!\!&
\mathcal{RBG} & \mbox{---} & \mbox{the variety of all regular bands
of groups.}
\end{array}
$$
\vskip0.2cm
\noindent
For the equality $\mathcal{LRO}\vee\mathcal{RRO} = \mathcal{RO}$ see [PR99], Theorem V.3.3.\\

Let $\mathcal{U} \in
\mathcal{L}(\mathcal{CS}),$ where $\mathcal{CS}$ is the variety of
all completely simple semigroups, be a non-orthodox variety.  Then we have $\mathcal{RO} \vee
\mathcal{U} \in (\mathcal{RO}, \mathcal{RBG}).$ Thus
$$
\mathcal{SG}^{T_{\ell}} \vee \mathcal{SG}^{T_r} = \mathcal{LRO} \vee \mathcal{RRO} \neq
\mathcal{RBG} = \mathcal{SG}^{T_{\ell}T_r} \cap \mathcal{SG}^{T_rT_{\ell}} \,.
$$
this is illustrated in Diagram 4.7.}
\end{example}

\[ \hspace*{-0.35cm}\begin{minipage}[t]{6cm}
\beginpicture
\setcoordinatesystem units <.7truecm,.7truecm>
\setplotarea x from 0 to 14, y from -1 to 5
\setlinear
\setdashes <1mm>
\plot 6 0  8.5 1.25 /
\plot 3.5 1.25  6 2.5 /

\plot 6 3.7  8.5 4.95 /


\plot 6.035 2.5  6.035 3.7 /


\setsolid

\plot 3.5 1.25  6 0 /
\plot 6 2.5  8.5 1.25 /

\plot 3.5 4.95  6 3.7 /


\plot 5.95 2.5  5.95 3.7 /


\setsolid
\put {\circle*{2.5}} at 6.05 0
\put {\circle*{2.5}} at 3.575 1.25
\put {\circle*{2.5}} at 8.575 1.25
\put {\circle*{2.5}} at 6.075 2.5

\put {\circle*{2.5}} at 6.075 3.7
\put {\circle*{2.5}} at 3.575 4.95
\put {\circle*{2.5}} at 8.575 4.95



\put {${\mathcal{SG}}$} at 6 -0.5
\put {${\mathcal{LRO} = \mathcal{SG}}^{T_{\ell}}$} at 1.6 1.25
\put {${\mathcal{L}}^{\ast}$} at 2.65 4.95
\put {${\mathcal{SG}}^{T_r} = \mathcal{RRO}$} at 10.5 1.25

\put {${\mathcal{RO}} = \mathcal{LRO} \vee {\mathcal{RRO}}$} at 9.2
2.5

\put {${\mathcal{L}}^{\ast} \cap {\mathcal{R}}^{\ast}$} at 8 3.7

\put {${\mathcal{R}}^{\ast}$} at 9.325 4.95



\endpicture
\end{minipage} \]
\vskip0.3cm

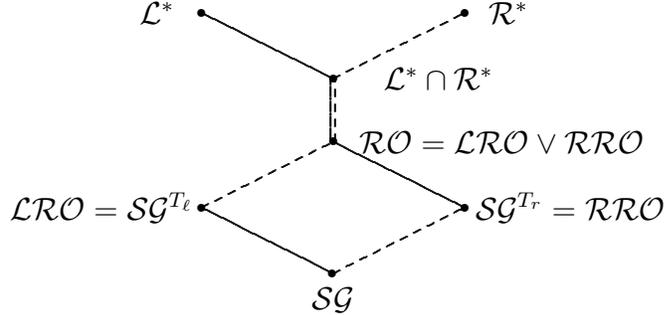
\centerline{\bf Diagram 4.7 \hspace*{1.4cm}}



We discuss briefly some ways in which the diagrams 4.5 and 4.7 can
be found in $\mathcal{L}(\mathcal{CR})$

\addtocounter{theorem}{1}

\begin{definition}
We refer to the variety $\mathcal{V}$ on which the Diagrams {\rm 4.2} and
{\rm 4.3} are built as the \emph{\bf base point} of the diagrams.
\end{definition}

For instance, we can start with $\mathcal{SG},$ $\mathcal{NBG}.$ Alternatively,
we can start with any self-dual proper subvariety $\mathcal{Z}$ of $\mathcal{CR}$,  pass to
$\mathcal{Z}^{K}$ which is also self-dual, and then pass to
$\mathcal{Z}^{KT}$ which 
is self-dual and a proper subvariety of $\mathcal{Z}^{KTK}$, and 
which also has the properties
required for a base point. For example, if we take $\mathcal{Z}
= \mathcal{T},\:\mathcal{G}$ or $\mathcal{CS}$ we obtain base points
$\mathcal{T}^{KT} = \mathcal{BG},$ $\mathcal{G}^{KT} = \mathcal{O}^T,$
$\mathcal{CS}^{KT} = L\mathcal{O}^T.$ Each of these is self-dual and
so we can repeat the process to obtain
$$
\mathcal{BG}^{KT}, \mathcal{BG}^{(KT)^2}, \ldots ,
\mathcal{BG}^{(KT)^n}, \mathcal{BG}^{(KT)^n},
(\mathcal{O}^T)^{(KT)^n}, (L\mathcal{O}^T)^{(KT)^n}\,.
$$

Another approach is to start with any variety $\mathcal{V} \in
{[\mathcal{S},\mathcal{CR})}.$ Then the varieties $\mathcal{V} \cap
\overline{\mathcal{V}}$ and $\mathcal{V} \vee \overline{\mathcal{V}}$
are both self-dual and so $(\mathcal{V} \cap \overline{\mathcal{V}})^T$
and $(\mathcal{V} \vee \overline{\mathcal{V}})^T$ are both base points
for diagrams and so on.




\section{Multiple copies of the lattice $[\mathcal{S}, \mathcal{B}]$}

Here we will be interested in the underlying abstract lattice of the lattice 
$[\mathcal{S}, \mathcal{B}]$ of varieties of bands containing the variety $\mathcal{S}$ 
of semilattices.  The sublattice $[\mathcal{S}, \mathcal{B}]$ is extremely well-known 
and has a very special role in the theory of semigroups.  In this section, we will 
show that copies of this lattice appear multiple times as a sublattice of various 
kernel classes of 
$\mathcal{L}(\mathcal{CR})$.

\begin{theorem}
Let $\mathcal{V, V}_{\ell}, \mathcal{V}_r \in [\mathcal{S}, \mathcal{CR}]$ be such that 
$\mathcal{V} \subset \mathcal{V}_{\ell} \subset \mathcal{V}^{K_{\ell}}$ and 
$\mathcal{V} \subset \mathcal{V}_r \subset \mathcal{V}^{K_r}$.   Let $\rho, \sigma, 
\tau \in \Theta^1$ be such that $h(\sigma) = T_r, h(\tau) = T_{\ell}$.  
Then the varieties $\mathcal{V}^{\rho(K_{\ell}, K_r)}, \mathcal{V}_{\ell}^{\sigma(K_{\ell}, K_r)}, 
\mathcal{V}_r^{\tau(K_{\ell}, K_r)}$, generate a sublattice $L$ of $\mathcal{V}K$ 
isomorphic to $[\mathcal{S}, \mathcal{B}]$.
\end{theorem}

{\bf{Note:}}  Varieties satisfying the  conditions in the hypothesis of Theorem 5.1 are not difficult to find.  For 
example, one could take $\mathcal{S, LNB, LRB}$ and their duals (in which case we obtain the 
interval $[\mathcal{S, B}]$ itself) 
or $\mathcal{SG, LNO, LRO}$ and their 
duals.  Also, the varieties $\mathcal{V}_{\ell}$ and  $\mathcal{V}_r$ may or may not be chosen 
independently of each other.  For instance, if $\mathcal{V}$ is self-dual, then one natural 
choice, after selecting $\mathcal{V}_{\ell}$, would be to take 
$\mathcal{V}_r = {\overline{\mathcal{V}}}_{\ell}$.  But other choices may be available, see 
Theorem 6.1.

\begin{proof}
The foundation of the lattice $L$ is shown in Diagram 5.2.\\

\[ \begin{minipage}[t]{12cm}
\beginpicture
\setcoordinatesystem units <0.8truecm,0.8truecm>
\setplotarea x from 0 to 12, y from -1 to 10
\setlinear


\setdashes <1mm> \setlinear

\setsolid\setlinear









\setdashes <1mm> \setlinear

\setsolid\setlinear





\setdashes <1mm> \setlinear

\setsolid\setlinear
\put {\circle*{2.5}} at 6.075 6.6


\put {${\mathcal{V}^{K_{\ell}}}\vee{\mathcal{V}}^{K_r}$} at 5.9 7.35


\plot 7.4 5.2  6 6.6 /

\plot 4.6 5.2  3.2 6.6 /


\setdashes <1mm> \setlinear
\plot 7.4 5.2  8.8 6.6 /
\plot 4.6 5.2  6 6.6 /

\setsolid\setlinear

\plot 6 1  4.6 2.4 /
\plot 7.4 2.4  6 3.8 /

\plot 4.6 2.4  3.2 3.8 /
\plot 6 3.8  4.6 5.2 /

\plot 8.8 3.8  7.4 5.2 /

\setdashes <1mm> \setlinear
\plot 6 3.8  7.4 5.2 /
\plot 7.4 2.4  8.8 3.8 /
\plot 3.2 3.8  4.6 5.2 /
\plot 4.6 2.4  6 3.8 /
\plot 6 1  7.4 2.4 /

\setsolid\setlinear
\put {\circle*{2.5}} at 6.075 1
\put {\circle*{2.5}} at 4.65 2.4
\put {\circle*{2.5}} at 7.475 2.4
\put {\circle*{2.5}} at 6.075 3.8

\put {\circle*{2.5}} at 3.25 3.8
\put {\circle*{2.5}} at 4.65 5.2
\put {\circle*{2.5}} at 3.25 6.6
\put {\circle*{2.5}} at 8.875 3.8
\put {\circle*{2.5}} at 7.475 5.2
\put {\circle*{2.5}} at 8.875 6.6

\put {$\mathcal{V}$} at 5.6 0.85
\put {$\mathcal{V}_{\ell}$} at 4.0 2.1
\put {$\mathcal{V}_r$} at 7.8 2.0
\put {$\mathcal{V}^{K_{\ell}}$} at 2.45 3.8
\put {$\mathcal{V}_{\ell}\vee \mathcal{V}_r$} at 7 3.7
\put {$\mathcal{V}^{K_r}$} at 9.6 3.8
\put {$\mathcal{V}_r^{K_{\ell}}$} at 2.215 6.6
\put {$\mathcal{V}_{\ell}^{K_r}$} at 9.825 6.6



\setdashes <1mm> \setlinear

\setsolid\setlinear

\endpicture
\end{minipage} \]
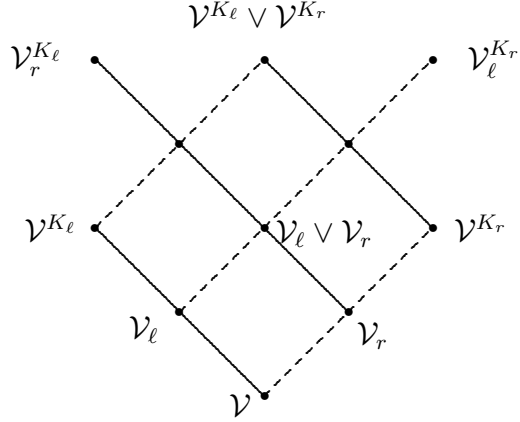
\centerline{\bf Diagram 5.2 \ }


\vskip0.2cm

This is simply the sublattice of $\mathcal{L}(\mathcal{CR})$ generated by the elements 
$\mathcal{V}, \mathcal{V}_{\ell}, \mathcal{V}^{K_{\ell}}, \mathcal{V}_r, \mathcal{V}^{K_r}$.  
However, we need to establish that all the elements are distinct and that the intersections are as 
indicated in the diagram.  Clearly, the joins are correctly shown.  We break the argument 
into convenient pieces.\\

\noindent
(1) $\mathcal{V}_{\ell}$ ({\em{respectively}}, $\mathcal{V}_r$) {\em{is not comparable with either}} 
$\mathcal{V}_r$ {\em{or}} $\mathcal{V}^{K_r}$ ({\em{respectively}}, $\mathcal{V}_{\ell}$ {\em{or}} $\mathcal{V}^{K_{\ell}}$). \\
We have 
$$
\mathcal{V}_{\ell} \subseteq \mathcal{V}^{K_r} \Longrightarrow \mathcal{V}_{\ell}\; K_r\; \mathcal{V} 
\Longrightarrow \mathcal{V}_{\ell}\;\; K_{\ell} \cap K_r\;\; \mathcal{V} \Longrightarrow 
\mathcal{V}_{\ell} = \mathcal{V}
$$
which is a contradiction.  A similar argument applies if we assume that $\mathcal{V}^{K_r} 
\subseteq \mathcal{V}_{\ell}$.  Thus $\mathcal{V}_{\ell}$ and $\mathcal{V}^{K_r}$ are 
incomparable.  
If $\mathcal{V}_{\ell}$ and $\mathcal{V}_r$ are comparable then, without loss of 
generality, we may assume that 
$\mathcal{V}_{\ell} \subseteq \mathcal{V}_r$.  This implies that $\mathcal{V}_{\ell}\;\; K_{\ell} \cap K_r\;\; \mathcal{V}$  
which is again a contradiction.  By duality, the claim holds.\\

\noindent
(2) $\mathcal{V}_{\ell} \vee \mathcal{V}_r$ {\em{is not comparable with}} $\mathcal{V}^{K_{\ell}}$ 
(respectively, $\mathcal{V}^{K_r}$) {\em{and is therefore distinct from}} $\mathcal{V}_{\ell}$ and 
$\mathcal{V}^{K_{\ell}}$ ({\em{respectively}} $\mathcal{V}_r$ and 
$\mathcal{V}^{K_r}$).\\

\noindent
That $\mathcal{V}_{\ell} \vee \mathcal{V}_r$ is not contained in $\mathcal{V}^{K_{\ell}}$ follows 
from (1).  On the other hand, by Lemmas 2.8(i) and 2.7(vi), 
$$
\mathcal{V}^{K_{\ell}} \subseteq \mathcal{V}_{\ell} \vee \mathcal{V}_r \Longrightarrow 
\mathcal{V}^{K_{\ell}} = (\mathcal{V}^{K_{\ell}})_{K_r} \subseteq (\mathcal{V}_{\ell} \vee 
\mathcal{V}_r)_{K_r} 
= (\mathcal{V}_{\ell})_{K_r} \vee \mathcal{V} 
\subseteq \mathcal{V}_{\ell}
$$
which is a contradiction.   This establishes the first claim and the remaining claims then follow.\\

\noindent
(3) ($\mathcal{V}_{\ell} \vee \mathcal{V}_r) \vee \mathcal{V}^{K_{\ell}} = \mathcal{V}_r \vee \mathcal{V}^{K_{\ell}}$ {\em{and is distinct from}} $\mathcal{V}_{\ell} \vee \mathcal{V}_r$ 
{\em{and}}  $\mathcal{V}^{K_{\ell}}$.\\
\noindent
The equality is clear and the remaining claims follow from the fact that $\mathcal{V}_{\ell} \vee \mathcal{V}_r$ and $\mathcal{V}^{K_{\ell}}$ are incomparable.\\

\noindent
(4)  $\mathcal{V}^{K_{\ell}}$ {\em{and}} $\mathcal{V}^{K_r}$ {\em{are incomparable and are distinct 
from}} $\mathcal{V}^{K_{\ell}} \vee \mathcal{V}^{K_r}$.\\
\noindent
If, for instance, $\mathcal{V}^{K_{\ell}} \subseteq \mathcal{V}^{K_r}$, then 
$\mathcal{V}_{\ell} \vee \mathcal{V}_r \subseteq \mathcal{V}^{K_r}$ which would contradict (2).  
Reversing the roles of $\mathcal{V}^{K_{\ell}}$ and $\mathcal{V}^{K_r}$ we see that 
these varieties are incomparable and therefore distinct from $\mathcal{V}^{K_{\ell}} \vee \mathcal{V}^{K_r}$.\\

\noindent
(5) $\mathcal{V}^{K_{\ell}} \vee \mathcal{V}_r$ {\em{and}} $\mathcal{V}_{\ell} \vee \mathcal{V}^{K_r}$ 
{\em{are incomparable and both are distinct from}} $\mathcal{V}^{K_{\ell}} \vee \mathcal{V}^{K_r}$.\\
\noindent
Suppose that 
$\mathcal{V}^{K_{\ell}} \vee \mathcal{V}_r \subseteq \mathcal{V}_{\ell} \vee \mathcal{V}^{K_r}$.  
Then by Lemmas 2.8(i) and 2.7(vi), 

\begin{eqnarray*}
\mathcal{V}^{K_{\ell}} = (\mathcal{V}^{K_{\ell}})_{K_r}
&\subseteq& (\mathcal{V}^{K_{\ell}} \vee \mathcal{V}_r)_{K_r} \subseteq (\mathcal{V}_{\ell} \vee \mathcal{V}^{K_r})_{K_r}\\
&=& (\mathcal{V}^{K_r})_{K_r} \vee (\mathcal{V}_{\ell})_{K_r} = \mathcal{V}_{K_r} \vee 
(\mathcal{V}_{\ell})_{K_r}\\
&\subseteq& \mathcal{V} \vee \mathcal{V}_{\ell} = \mathcal{V}_{\ell}
\end{eqnarray*}
which is a contradiction.  The assumption of the reverse containment also leads to a 
contradiction.  Accordingly, the first claim holds.  The second claim follows immediately 
from the first.\\

\noindent
(6)  {\em{The following are sets of}} $K_{\ell}$-{\em{related varieties}}
$$
\{\mathcal{V}_r, \mathcal{V}_{\ell} \vee \mathcal{V}_r, \mathcal{V}^{K_{\ell}} \vee \mathcal{V}_r, 
\mathcal{V}_r^{K_{\ell}}\}, \;\{\mathcal{V}^{K_r}, \mathcal{V}_{\ell} \vee \mathcal{V}^{K_r}, 
\mathcal{V}^{K_{\ell}} \vee \mathcal{V}^{K_r}\}
$$
{\em{and the following are sets of}} $K_r$-{\em{related varieties}}
$$
\{\mathcal{V}_{\ell}, \mathcal{V}_{\ell} \vee \mathcal{V}_r, \mathcal{V}_{\ell} \vee \mathcal{V}^{K_r}, 
\mathcal{V}_{\ell}^{K_r}\}, \;\{\mathcal{V}^{K_{\ell}}, \mathcal{V}^{K_{\ell}} \vee \mathcal{V}_r, \mathcal{V}^{K_{\ell}} \vee \mathcal{V}^{K_r}\}.
$$
\noindent
These claims follow immediately from the choices of $\mathcal{V}_{\ell}$ and $\mathcal{V}_r$ 
together with the fact that $K_{\ell}$ and $K_r$ are complete congruences on 
$\mathcal{L}(\mathcal{CR})$.\\

\noindent
(7) $\mathcal{V}_{\ell} \vee \mathcal{V}^{K_r} \subset \mathcal{V}_{\ell}^{K_r}$ {\em{and}} 
$\mathcal{V}^{K_{\ell}} \vee \mathcal{V}_r \subset \mathcal{V}_r^{K_{\ell}}$.\\
\noindent
Clearly it suffices to establish just one of these containments.  It is also clear that 
$\mathcal{V}_{\ell} \vee \mathcal{V}^{K_r} \subseteq \mathcal{V}_{\ell}^{K_r}$ so that 
the goal is to establish that the containment is proper.  Let $F$ denote the free object 
in $\mathcal{V}_{\ell}^{K_r}$ on a countably infinite set $X$.  Let $R = F^1$ as a set and endowed 
with the right zero multiplication.  Let $S$ denote the disjoint union $F \cup R$ of these 
sets and define a product $\ast$ on $S$ as follows.  The operation $\ast$ agrees with the 
already defined multiplication in $F$ and in $R$ while, for $a \in F, b\in R$, we  have 
$$
a \ast b = b, \;\;\mbox{and} \;\; b\ast a = ba
$$
where $ba$ denotes the product within $F^1$, as the semigroup $F$ with an identity adjoined.  
It is routine to check that this multiplication is associative and that $S$ is a union of groups.  
Therefore $S$ is a completely regular semigroup.  \\

The Rees quotient $S/R$ is isomorphic to $F^0$ and so belongs to $\mathcal{V}_{\ell}^{K_r}$.  
Hence $S \in (\mathcal{V}_{\ell}^{K_r})^{K_r} = \mathcal{V}_{\ell}^{K_r}$.  Now assume, 
by way of contradiction, that $\mathcal{V}_{\ell}^{K_r} \subseteq \mathcal{V}_{\ell} \vee 
\mathcal{V}^{K_r}$.  Since $\mathcal{V}_{\ell} \vee \mathcal{V}^{K_r}$ is $K_{\ell}$-related 
to $\mathcal{V}^{K_r}$, we must have that $S \in (\mathcal{V}^{K_r})^{K_{\ell}}$ which, 
by Lemma 2.5, means that $S/(\tau \cap \mathcal{L}^0) \in \mathcal{V}^{K_r}$ where 
$\tau$ denotes the largest idempotent pure congruence on $S$. It is helpful to note that 
$\tau \cap \mathcal{L}^0 = (\tau \cap \mathcal{L})^0$.\\  

Let $a, b \in S$ be such that $a\; \mathcal{L}^0\; b$.  Since $R$ is a right zero semigroup 
and $a\;\mathcal{L}\; b$ either 
$a = b \in R$ or $a, b \in F$.  If $a, b \in R$ then trivially $a = b$.  If $a, b \in F$ then $1\ast a\; 
\mathcal{L}^0\; 1\ast b$ so that, again since $R$ is a right zero semigroup,
 $1\ast a = 1\ast b$ and $a = b$.  Consequently, $\mathcal{L}^0 = 
\epsilon $, the identity congruence on $S$, and also $\tau \cap \mathcal{L}^0 = \epsilon$.  
This implies that $S \in \mathcal{V}^{K_r}$.  But $F$ is a subsemigroup of $S$.  Hence 
$F \in \mathcal{V}^{K_r}$ which implies that $\mathcal{V}_{\ell}^{K_r} \subseteq 
\mathcal{V}^{K_r}$ and therefore that $\mathcal{V}_{\ell} \subseteq \mathcal{V}^{K_r}$, 
contradicting (1).  Therefore the claim holds.\\

\noindent
(8)  $\mathcal{V}_{\ell}^{K_r}, \mathcal{V}_r^{K_{\ell}}$ {\em{and}} $\mathcal{V}^{K_{\ell}} \vee 
\mathcal{V}^{K_r}$ {\em{are pairwise incomparable}}. \\
\noindent
If $\mathcal{V}_{\ell}^{K_r} \subseteq \mathcal{V}_r^{K_{\ell}}$ then
$$
\mathcal{V}_{\ell} \subseteq \mathcal{V}_{\ell}^{K_r} = (\mathcal{V}_{\ell}^{K_r})_{K_{\ell}} 
\subseteq (\mathcal{V}_{r}^{K_{\ell}})_{K_{\ell}} = (\mathcal{V}_r)_{K_{\ell}}  \subseteq \mathcal{V}_r
$$ 
contradicting (1).  Reversing the roles of $\mathcal{V}_{\ell}^{K_r}$ and $\mathcal{V}_{r}^{K_{\ell}}$, 
we see that these varieties are incomparable.  Next 
\begin{eqnarray*}
\mathcal{V}^{K_{\ell}} \vee \mathcal{V}^{K_r} \subseteq \mathcal{V}_{\ell}^{K_r} 
&\Longrightarrow& (\mathcal{V}^{K_{\ell}} \vee \mathcal{V}^{K_r})_{K_r} \subseteq 
(\mathcal{V}_{\ell}^{K_r})_{K_r} = (\mathcal{V}_{\ell})_{K_r} \subseteq \mathcal{V}_{\ell}\\
&\Longrightarrow& (\mathcal{V}^{K_{\ell}})_{K_r} \vee (\mathcal{V}^{K_r})_{K_r} \subseteq 
\mathcal{V}_{\ell} \;\;\\
&\Longrightarrow& \mathcal{V}^{K_{\ell}} \vee \mathcal{V}_{K_r} \subseteq 
\mathcal{V}_{\ell} \;\;\\
&\Longrightarrow& \mathcal{V}^{K_{\ell}} \subseteq \mathcal{V}_{\ell}
\end{eqnarray*}
contradicting the choice of $\mathcal{V}_{\ell}$.  On the other hand, if 
$\mathcal{V}_{\ell}^{K_r} \subseteq \mathcal{V}^{K_{\ell}} \vee \mathcal{V}^{K_r}$, then 
$$
\mathcal{V}_{\ell}^{K_r} = (\mathcal{V}_{\ell}^{K_r})_{K_{\ell}} \subseteq 
(\mathcal{V}^{K_{\ell}} \vee \mathcal{V}^{K_r})_{K_{\ell}} = \mathcal{V}_{K_{\ell}} \vee \mathcal{V}^{K_r} 
= \mathcal{V}^{K_r}
$$
which contradicts (7).  Hence $\mathcal{V}_{\ell}^{K_r}$ and $\mathcal{V}^{K_{\ell}} \vee \mathcal{V}^{K_r}$ are incomparable.  The dual argument applies to $\mathcal{V}_r^{K_{\ell}}$ 
and $\mathcal{V}^{K_{\ell}} \vee \mathcal{V}^{K_r}$.\\

From (6), we know that the varieties 
$\mathcal{V}_{\ell} \vee \mathcal{V}_r, \mathcal{V}^{K_{\ell}} \vee \mathcal{V}_r $ 
and 
$\mathcal{V}_r^{K_{\ell}}$ are $K_{\ell}$-related and by (3), (7), that 
$$
\mathcal{V}_{\ell} \vee \mathcal{V}_r \subset \mathcal{V}^{K_{\ell}} \vee \mathcal{V}_r
\subset \mathcal{V}_r^{K_{\ell}}.
$$
Similarly, the varieties $\mathcal{V}_{\ell} \vee \mathcal{V}_r, \mathcal{V}_{\ell} \vee 
\mathcal{V}^{K_{r}}$ and 
$\mathcal{V}_{\ell}^{K_r}$ are $K_r$-related and 
$$
\mathcal{V}_{\ell} \vee \mathcal{V}_r \subset \mathcal{V}_{\ell} \vee \mathcal{V}^{K_r} 
\subset \mathcal{V}_{\ell}^{K_r}.
$$
Thus we may repeat the above discussion using $\mathcal{V}_{\ell} \vee \mathcal{V}_r$ 
as our starting point with $\mathcal{V}_{\ell} \vee \mathcal{V}_r, \mathcal{V}_r \vee \mathcal{V}^{K_{\ell}}, \mathcal{V}_{\ell} \vee \mathcal{V}^{K_r}$ replacing $\mathcal{V}, 
\mathcal{V}_{\ell}, \mathcal{V}_r$, respectively, to obtain another nine-element 
sublattice which overlaps with the lattice in Diagram 5.2 to obtain the following 
larger sublattice of $\mathcal{V}K$.  

\vskip-0.5cm

\[ \begin{minipage}[t]{12cm}
\beginpicture
\setcoordinatesystem units <0.8truecm,0.8truecm>
\setplotarea x from 0 to 12, y from -1 to 11
\setlinear


\setdashes <1mm> \setlinear

\setsolid\setlinear



\put {\circle*{1.5}} at 6.075 11
\put {\circle*{1.5}} at 6.075 11.2
\put {\circle*{1.5}} at 6.075 11.4
\put {\circle*{1.5}} at 6.075 11.6
\put {\circle*{1.5}} at 6.075 11.8


\put {$\mathcal{V}^K$} at 6.05 13.5
\put {\circle*{1.5}} at 6.05 13.0




\setdashes <1mm> \setlinear

\setsolid\setlinear


\plot 7.4 8  6 9.4 /



\setdashes <1mm> \setlinear
\plot 4.6 8  6 9.4 /

\setsolid\setlinear
\put {\circle*{2.5}} at 6.075 6.6
\put {\circle*{2.5}} at 6.075 9.4


\put {${\mathcal{V}_r^{K_{\ell}}}\vee{\mathcal{V}}_{\ell}^{K_r}$} at 6.0 10
\put {${\mathcal{V}_r^{K_{\ell}}}\vee{\mathcal{V}}^{K_r}$} at 3.5 8.35
\put {${\mathcal{V}^{K_{\ell}}}\vee{\mathcal{V}}_{\ell}^{K_r}$} at 8.7 8.3
\put {${\mathcal{V}^{K_{\ell}}}\vee{\mathcal{V}}^{K_r}$} at 5.9 7.35

\plot 4.6 8  3.1 9.5 /
\put {\circle*{2.5}} at 3.15 9.52
\put {$(\mathcal{V}_{\ell}\vee \mathcal{V}^{K_r})^{K_{\ell}} = \mathcal{V}^{K_rK_{\ell}}$} at 0 9.55 

\plot 6 6.6  4.6 8 /

\plot 7.4 5.2  6 6.6 /

\plot 4.6 5.2  3.2 6.6 /

\plot 8.8 6.6  7.4 8 /

\setdashes <1mm> \setlinear

\plot 7.4 8   8.8  9.5 /
\put {\circle*{2.5}} at 8.83 9.5
\put {$\mathcal{V}^{K_{\ell}K_{r}} =
(\mathcal{V}^{K_{\ell}}\vee \mathcal{V}_r)^{K_{r}}$} at 12 9.55

\plot 7.4 5.2  8.8 6.6 /
\plot 3.2 6.6  4.6 8 /
\plot 4.6 5.2  6 6.6 /
\plot 6 6.6  7.4 8 /

\setsolid\setlinear
\put {\circle*{2.5}} at 4.65 8
\put {\circle*{2.5}} at 7.475 8

\plot 6 1  4.6 2.4 /
\plot 7.4 2.4  6 3.8 /

\plot 4.6 2.4  3.2 3.8 /
\plot 6 3.8  4.6 5.2 /

\plot 8.8 3.8  7.4 5.2 /

\setdashes <1mm> \setlinear
\plot 6 3.8  7.4 5.2 /
\plot 7.4 2.4  8.8 3.8 /
\plot 3.2 3.8  4.6 5.2 /
\plot 4.6 2.4  6 3.8 /
\plot 6 1  7.4 2.4 /

\setsolid\setlinear
\put {\circle*{2.5}} at 6.075 1
\put {\circle*{2.5}} at 4.65 2.4
\put {\circle*{2.5}} at 7.475 2.4
\put {\circle*{2.5}} at 6.075 3.8

\put {\circle*{2.5}} at 3.25 3.8
\put {\circle*{2.5}} at 4.65 5.2
\put {\circle*{2.5}} at 3.25 6.6
\put {\circle*{2.5}} at 8.875 3.8
\put {\circle*{2.5}} at 7.475 5.2
\put {\circle*{2.5}} at 8.875 6.6

\put {$\mathcal{V}$} at 5.6 0.85
\put {$\mathcal{V}_{\ell}$} at 4.0 2.1
\put {$\mathcal{V}_r$} at 7.8 2.0
\put {$\mathcal{V}^{K_{\ell}}$} at 2.45 3.8
\put {$\mathcal{V}^{K_{\ell}}\vee \mathcal{V}_r$} at 3.2 5.2
\put {$\mathcal{V}_{\ell}\vee \mathcal{V}^{K_r}$} at 8.7 5.2
\put {$\mathcal{V}_{\ell}\vee \mathcal{V}_r$} at 7 3.7
\put {$\mathcal{V}^{K_r}$} at 9.6 3.8
\put {$\mathcal{V}_r^{K_{\ell}}$} at 2.215 6.6

\put {$\mathcal{V}_{\ell}^{K_r}$} at 9.825 6.6



\setdashes <1mm> \setlinear

\setsolid\setlinear


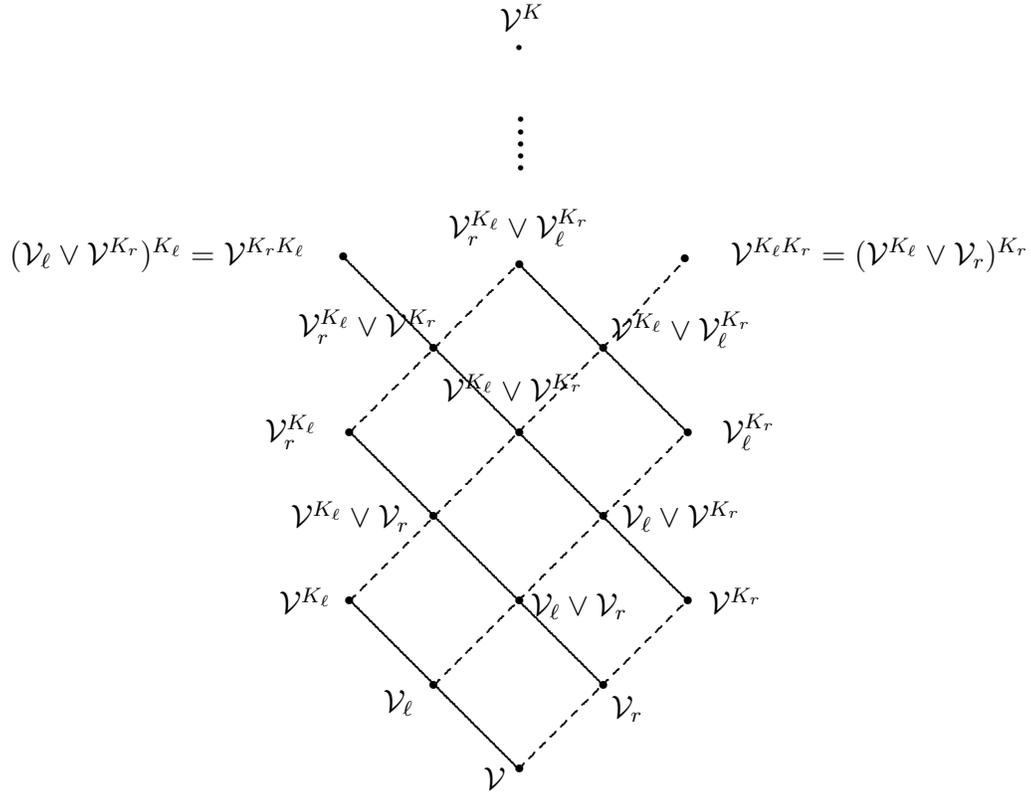
\put {{\bf Diagram 5.3} \ } at 6 0
\endpicture
\end{minipage} \]

Regarding the varieties top left and top right in the main body of the diagram, since 
$\mathcal{V}_{\ell} \vee \mathcal{V}^{K_r}$ is $K_{\ell}$-related to $\mathcal{V}^{K_r}$ we must have 
$(\mathcal{V}_{\ell} \vee \mathcal{V}^{K_r})^{K_{\ell}} = \mathcal{V}^{K_rK_{\ell}}$ and dually 
$(\mathcal{V}^{K_{\ell}}\vee \mathcal{V}_r)^{K_r} = \mathcal{V}^{K_{\ell}K_r}$.

Each new level provides the foundation for the next level and so on, thereby generating 
a lattice of subvarieties of $\mathcal{V}K$ which is isomorphic to the interval $[\mathcal{S}, \mathcal{B})$.  Then by Lemma 2.8(ii), the supremum of all the elements in 
this lattice is 
$$
\bigvee_{u \in \Theta} \mathcal{V}^{u(K_{\ell}, K_r)} = \mathcal{V}^K.
$$
Therefore we may legitimately adjoin $\mathcal{V}^K$ to the top of the Diagram 5.3 to 
obtain a complete sublattice of $\mathcal{V}K$ that is isomorphic to the interval 
$[\mathcal{S}, \mathcal{B}]$.  Note that we have also proved that the lines of positive slope 
connect $K_r$-related varieties and the lines of negative slope connect varieties that are 
$K_{\ell}$-related.  
\end{proof}

The most important illustration of Theorem 5.1 and the only previously known example of the 
behaviour described there has been the lattice $[\mathcal{S}, \mathcal{B}]$ itself which can be 
viewed as illustrating Theorem 5.1 by taking $\mathcal{V} = \mathcal{S}, \mathcal{V}_{\ell} = 
\mathcal{LNB}, \mathcal{V}_r = \mathcal{RNB}$.


\section{More copies of the lattice $[\mathcal{S}, \mathcal{B}]$}

In this section we indicate ways in which the conditions in Theorem 5.1 can be relaxed.

\begin{theorem}
Let $\mathcal{V, V}_{\ell}, \mathcal{V}^{\ell}, \mathcal{V}_r, \mathcal{V}^r  \in [\mathcal{S}, \mathcal{CR}]$ be such that 
$\mathcal{V} \subset \mathcal{V}_{\ell} \subset \mathcal{V}^{\ell} \subseteq \mathcal{V}^{K_{\ell}}$ and 
$\mathcal{V} \subset \mathcal{V}_r \subset \mathcal{V}^r \subseteq \mathcal{V}^{K_r}$.   In addition, assume 
that $(\mathcal{V}^{\ell})_{K_r} = \mathcal{V}^{\ell}$ and $(\mathcal{V}^r)_{K_{\ell}} = \mathcal{V}^r$.  
Then we have the following:\\
{\rm(i)} $\mathcal{V}\; K_{\ell}\;  \mathcal{V}_{\ell}\;  K_{\ell}\; \mathcal{V}^{\ell}$ and $\mathcal{V}\; K_{r}\; \mathcal{V}_r\; K_r\; \mathcal{V}^r$.\\
{\rm(ii)} The varieties
$$
\mathcal{V},  \mathcal{V}_{\ell}, \mathcal{V}^{\ell}, \mathcal{V}_r, \mathcal{V}^r, \mathcal{V}_{\ell} \vee \mathcal{V}_r, 
\mathcal{V}_{\ell} \vee \mathcal{V}^r, \mathcal{V}^{\ell} \vee \mathcal{V}_r, \mathcal{V}^{\ell} \vee \mathcal{V}^r
$$
constitute a sublattice of nine distinct elements in $\mathcal{L}(\mathcal{CR})$ as depicted in Diagram 6.2.\\
{\rm(iii)}  $\mathcal{V}^{\ell} \vee \mathcal{V}_r \subset \mathcal{V}_r^{K_{\ell}}$ and 
$\mathcal{V}_{\ell} \vee \mathcal{V}^r \subset \mathcal{V}_{\ell}^{K_r}$. \\
{\rm(iv)}  By selecting varieties $(\mathcal{V}_{\ell})^r$ and $(\mathcal{V}_{r})^{\ell}$ such that
$$
\mathcal{V}_{\ell}\vee \mathcal{V}^r \subset (\mathcal{V}_{\ell})^r \subseteq (\mathcal{V}_{\ell})^{K_r}\; 
\mbox{and}\;  
((\mathcal{V}_{\ell})^r)_{K_{\ell}} = (\mathcal{V}_{\ell})^r
$$
and
$$
\mathcal{V}^{\ell}\vee \mathcal{V}_r 
\subset (\mathcal{V}_{r})^{\ell} \subseteq (\mathcal{V}_{r})^{K_{\ell}}\; \mbox{and}\; 
((\mathcal{V}_{r})^{\ell})_{K_r} = (\mathcal{V}_{r})^{\ell},
$$
the procedure in parts (i) - (iii) may now be repeated starting from the base consisting 
of the varieties 
$$
\mathcal{V}_{\ell}\vee \mathcal{V}_r, \; \mathcal{V}^{\ell}\vee \mathcal{V}_r, \; (\mathcal{V}_r)^{\ell}, \; 
\mathcal{V}_{\ell}\vee \mathcal{V}^r, \; (\mathcal{V}_{\ell})^r.
$$
Repeating the process and including $\mathcal{V}^K$ yields a sublattice (but not necessarily a 
complete sublattice) of $\mathcal{V}K$ isomorphic to $[\mathcal{S}, \mathcal{B}]$.  
\end{theorem}

\noindent
Note also that the varieties $\mathcal{V}_{\ell}$ and  $\mathcal{V}_r$ may or may not be chosen 
independently.  Once again, if $\mathcal{V}$ is self-dual, then one natural choice, after selecting $\mathcal{V}_{\ell}$, would be to take 
$\mathcal{V}_r = {\overline{\mathcal{V}}}_{\ell}$.  Similarly, one might choose $\mathcal{V}^{\ell} = 
\mathcal{V}^{K_{\ell}}$, but other choices may be available.  The same applies to $\mathcal{V}^r$.  
For instance, one could choose 
$\mathcal{SG, LNO, LRO}$, $\mathcal{RNO, LRB} \vee \mathcal{G}$.\\

\begin{proof}  The proof follows the lines of the proof of Theorem 5.1. 

 
\end{proof}

It is natural to wonder the extent to which different lattices constructed as above starting from 
the same base variety $\mathcal{V}$ might overlap.  The following simple observation sheds 
some light on that, especially for what might be called the {\em{default}} option after the 
choice of the starting varieties $\mathcal{V}, \mathcal{V}_{\ell}, \mathcal{V}^{\ell}, \mathcal{V}_r, 
\mathcal{V}^r$, namely where we always choose the largest element in each $K_{\ell}$-class and 
each $K_{r}$-class.  This leads to the lattice consisting of all the elements of the form 
$\mathcal{V}, (\mathcal{V}_{\ell})^u, (\mathcal{V}^{\ell})^u, (\mathcal{V}_r)^v, 
(\mathcal{V}^r)^v$, where $u = u(K_{\ell},K_r), v = v(K_{\ell}, K_r), h(u) = K_r, h(v) = K_{\ell}$ 
and their intersections.

\begin{lemma}
Let $\mathcal{U, V} \in [\mathcal{S}, \mathcal{CR}], \mathcal{U} \; K_{\ell} \;  \mathcal{V}, \; 
\mathcal{U} \neq \mathcal{V}$.  Let $u = u(K_{\ell}, K_r)$ be such that $h(u) = T_r$.  Then 
$\mathcal{U}^u \neq \mathcal{V}^u$.
\end{lemma}

\begin{proof}
We argue by induction on $\mid u \mid$.   Since we know that $\mathcal{U} \; K_{\ell} \;  \mathcal{V}, \; 
\mathcal{U} \neq \mathcal{V}$ and that $K_{\ell} \cap K_r = \epsilon$, it follows that $\mathcal{U}$  and 
$\mathcal{V}$ are not $K_r$-related.  Hence $\mathcal{U}^{K_r} \neq \mathcal{V}^{K_r}$ and the claim 
holds for $\mid u \mid = 1$.  Now assume that the claim is true for all words of shorter length than $u$ 
and that $\mid u \mid > 1$.  
Without loss of generality, we may assume that $t(u) = K_{\ell}$ and that $u = vK_{\ell}$ for suitable $v$ 
with $t(v) = K_r$.  
Suppose that $\mathcal{U}^u = \mathcal{V}^u$.  We must have $t(v) = K_r$ so that 
$$
\mathcal{U}^v = (\mathcal{U}^v)_{K_{\ell}} = ((\mathcal{U}^v)^{K_{\ell}})_{K_{\ell}} = 
(\mathcal{U}^u)_{K_{\ell}} = (\mathcal{V}^u)_{K_{\ell}} = 
((\mathcal{V}^v)^{K_{\ell}})_{K_{\ell}} = (\mathcal{V}^v)_{K_{\ell}} = \mathcal{V}^v.
$$
Since $\mid v \mid < \mid u \mid$, this contradicts the induction hypothesis and 
therefore $\mathcal{U}^u \neq \mathcal{V}^u$ as required.
\end{proof}


We know that $\mid \mathcal{V}K_{\ell}\mid = 3, 4\;\mbox{or}\; 5$ for all 
$\mathcal{V} \in [\mathcal{S}, \mathcal{B}]$ 
and Theorem 6.1 applies only in the context of $K_{\ell}$ and $K_r$ classes containing at least three 
elements.  On the other hand, the cardinality of $\mathcal{SG}K_{\ell}$ is $2^{{\aleph}_0}$.  Recall 
the definition and basic properties of the variety $\mathcal{LRO}$of left regular orthogroups from [PR99].  We conclude with an analysis of the $K_{\ell}$-class of $\mathcal{SG}$.  This has some 
interesting features.

\begin{lemma}  Let $\mathcal{V} \in \mathcal{SG}K_{\ell}$. 
\begin{enumerate}
\item[{\rm (i)}] $\mathcal{SG}K_{\ell} = [\mathcal{SG}, \mathcal{LRO}]$. 
\item[{\rm (ii)}]  $\mathcal{V} \cap \mathcal{B} \in \{\mathcal{S, LNB, LRB}\}$.
\item[{\rm (iii)}]  $\mathcal{V} \cap \mathcal{B} = \mathcal{S} \Longleftrightarrow \mathcal{V} 
= \mathcal{SG}$.
\item[{\rm (iv)}]  $\mathcal{V} \cap \mathcal{B} = \mathcal{LNB} \Longleftrightarrow \mathcal{V} 
= \mathcal{LNO}$.
\item[{\rm (v)}]  $\mathcal{V} \cap \mathcal{B} = \mathcal{LRB} \Longleftrightarrow \mathcal{V} 
\in [\mathcal{LRB\vee \mathcal{G}, LRO}]$.
\item[{\rm (vi)}]   $[\mathcal{LRB\vee \mathcal{G}, LRO}] \cong \mathcal{L}(\mathcal{G})$.
\end{enumerate}
\end{lemma}
\begin{proof}
(i)  See [RK1] Theorem 6.3(iv).\\
(ii)  We have
$$
\mathcal{S} = \mathcal{SG} \cap \mathcal{B} \subseteq \mathcal{V} \cap \mathcal{B} 
\subseteq \mathcal{LRO} \cap \mathcal{B} = \mathcal{LRB}
$$
where $[\mathcal{S, LRB}] = \{\mathcal{S, LNB, LRB}\}$.
Therefore the claim holds.\\
(iii)  By [PR99] Theorem IV.2.4, if $\mathcal{V} \cap \mathcal{B} = \mathcal{S}$ then 
$\mathcal{V} \subseteq \mathcal{SG}$.  By the hypothesis and part (i), $\mathcal{SG} \subseteq 
\mathcal{V}$ so that equality prevails.  The converse implication is trivial.\\
(iv)  By [PR99] Corollary IV.2.12, if $\mathcal{V} \cap \mathcal{B} = \mathcal{LNB}$ then 
$\mathcal{V} \subseteq \mathcal{LNO}$ so that $\mathcal{LNB} \subseteq 
\mathcal{V} \subseteq \mathcal{LNO}$.  On the other hand, by [PR99] Corollary IV.2.12, 
$\mathcal{LNO} = \mathcal{S} \vee \mathcal{LZ} \vee \mathcal{G} \subseteq \mathcal{V}$.  Hence 
$\mathcal{V} = \mathcal{LNO}$.  The converse is clear.\\
(v)  We have
\begin{eqnarray*}
\mathcal{V} \in \mathcal{SG}K_{\ell}, \mathcal{V}\cap\mathcal{B} = \mathcal{LRB} 
&\Longrightarrow& \mathcal{LRB, SG} \subseteq \mathcal{V} \subseteq 
\mathcal{SG}^{K_{\ell}}\\
&\Longrightarrow&  \mathcal{V} \in [\mathcal{LRB} \vee \mathcal{G}, \mathcal{LRO}].
\end{eqnarray*}

\noindent
Conversely, let $\mathcal{V} \in [\mathcal{LRB} \vee \mathcal{G}, \mathcal{LRO}]$.  Then 
$$
\mathcal{SG} \subseteq \mathcal{V} \subseteq \mathcal{LRO} = \mathcal{SG}^{K_{\ell}}
$$
so that $\mathcal{V} \in \mathcal{SG}K_{\ell}$.  In addition,
$$
\mathcal{LRB} \subseteq \mathcal{V} \cap \mathcal{B} \subseteq \mathcal{LRO} \cap 
\mathcal{B} = \mathcal{LRB}.
$$
Thus $\mathcal{V} \cap \mathcal{B} = \mathcal{LRB}$ and the claim holds.\\
(vi)  See Reilly [R1], Theorem 6.3(iv).
\end{proof}


\vskip1.6cm

\noindent
{\small
Department of Mathematics \\
Simon Fraser University \\
Burnaby, British Columbia\\
Canada V5A 1S6 \\[0.2cm]
Email: nreilly@sfu.ca
}
\end{document}